\documentclass[11pt,reqno]{amsart}

\usepackage{amsmath,amssymb,amsfonts,amsthm,mathtools}
\usepackage[margin=1.3in]{geometry}
\usepackage{enumitem}
\usepackage{hyperref}
\usepackage{microtype}
\usepackage{comment}

\hypersetup{
    colorlinks=true,
    linkcolor=blue,
    citecolor=magenta,
    urlcolor=blue
}
\numberwithin{equation}{section}

\newtheorem{theorem}{Theorem}[section]
\newtheorem{proposition}[theorem]{Proposition}
\newtheorem{lemma}[theorem]{Lemma}
\newtheorem{corollary}[theorem]{Corollary}

\theoremstyle{definition}

\newtheorem{remark}[theorem]{Remark}

\newcommand{\R}{\mathbb R}

\newcommand{\Sn}{\mathbb S}
\newcommand{\Om}{\Omega}
\newcommand{\Sig}{\Sigma}

\newcommand{\dd}{\,d}
\newcommand{\calH}{\mathcal H}

\newcommand{\diver}{\operatorname{div}}

\newcommand{\HH}{\mathbb H}

\allowdisplaybreaks

\title[Weinstock inequalities]{Weinstock inequalities for outward-minimizing domains}

\author[C. Gao]{Chaoqun Gao}
\author[Y. Wei]{Yong Wei}
\author[R. Zhou]{Rong Zhou}
\address{School of Mathematical Sciences, University of Science and Technology of China, Hefei 230026, P.R. China}
\email{\href{mailto:gaochaoqun@mail.ustc.edu.cn}{gaochaoqun@mail.ustc.edu.cn}}
\email{\href{mailto:yongwei@ustc.edu.cn}{yongwei@ustc.edu.cn}}
\email{\href{mailto:zhourong@mail.ustc.edu.cn}{zhourong@mail.ustc.edu.cn}}

\begin{document}


\begin{abstract}
We prove sharp Weinstock inequalities for the first nonzero Steklov eigenvalue of smooth outward-minimizing domains in Euclidean space and hyperbolic space.  The method is based on the weak inverse mean curvature flow of Huisken--Ilmanen.  The main new ingredient is an endpoint distributional monotonicity argument obtained from the calibrated weak formulation and the Gauss--Green formula for divergence-measure fields.
\end{abstract}

\maketitle


\section{Introduction}

Let $(M^n,g)$ be either the Euclidean space $\R^n$ or the hyperbolic space $\HH^n$, and let $\Omega\subset M$ be a bounded connected domain with smooth boundary $\Sigma=\partial\Omega$.  The Steklov eigenvalue problem is
\begin{equation*}
\begin{cases}
    \Delta u=0 & \text{in }\Omega,\\
    \partial_\nu u=\sigma u & \text{on }\Sigma,
\end{cases}
\end{equation*}
where $\Delta$ is the Laplace operator and $\nu$ denotes the outward unit normal.  The first nonzero Steklov eigenvalue is characterized by
\begin{equation}\label{eq:steklov-variational}
    \sigma_1(\Omega)
    =\inf\left\{
        \frac{\int_\Omega |\nabla u|^2\,dv}{\int_{\Sigma} u^2\,d\mu}:
        u\in H^1(\Omega)\setminus\{0\},\quad
        \int_{\Sigma}u\,d\mu=0
      \right\}.
\end{equation}
Here and below $\sigma_1$ denotes the first positive Steklov eigenvalue.  The Steklov problem was introduced by Vladimir Steklov at the turn of the $20$th century.  It is equivalently the spectral problem for the Dirichlet-to-Neumann map, and it arises in inverse problems, hydrodynamics, free boundary minimal surface theory, and differential geometry.  We refer to \cite[Chapter~7]{LMP2023} for an introduction to the Steklov problem and to \cite{CGGS24,Fra20,GP2017} for surveys on recent progress.

One of the central themes in Steklov spectral geometry is the search for isoperimetric upper bounds.  In dimension two, Weinstock \cite{W1954} proved that among simply connected planar domains with prescribed boundary length the disk maximizes $\sigma_1$.  The simple connectivity hypothesis is essential in the plane, and Girouard--Polterovich  \cite[Open Problem~2]{GP2017} ask for the maximal value of $\sigma_1$ among Euclidean domains with prescribed perimeter and for the domains, or limiting sequences of domains, on which it is realized.  In higher Euclidean dimensions, Fraser--Schoen  \cite{FS2019} showed that the direct analogue of Weinstock's theorem is false even among smooth contractible domains.  The recent survey \cite[Section~4]{CGGS24} emphasizes this obstruction and asks for geometric settings in which sharp ball-type inequalities can still be recovered.

Bucur--Ferone--Nitsch--Trombetti \cite{BFNT2021} proved the sharp Euclidean inequality in the class of convex domains.  Kwong--Wei \cite{KwongWei2023} later weakened convexity to the star-shaped mean-convex class by using a weighted three-term isoperimetric inequality \cite{GR2020} with a smooth inverse mean curvature flow argument.  Their theorem fits into the broader program of understanding which geometric assumptions can replace convexity in higher-dimensional Steklov inequalities.  

In space forms, \cite[Open Question~4.27]{CGGS24} asks in particular whether the convex-domain Weinstock inequality has hyperbolic and spherical analogues.  In hyperbolic space, Gu--Li--Wan \cite{GuLiWan2025} proved this for smooth star-shaped mean-convex domains in $\HH^n$, $n\geq4$.  Gu's thesis \cite{GuThesis2026} (see also \cite{GLW26}) contains the additional low-dimensional case $n=3$.  Their proof again uses smooth inverse mean curvature flow, together with a special volume-area ratio function $g$ and an auxiliary function $h$.

The purpose of this paper is to extend both results \cite{KwongWei2023,GuLiWan2025,GuThesis2026} from star-shaped mean-convex domains to outward-minimizing domains.  For a finite-perimeter set $E$ we write $P(E)=\calH^{n-1}(\partial^*E)$.  See \cite[Chapters~12 and~15]{Maggi2012} for the Euclidean theory of finite-perimeter sets and reduced boundaries.  A bounded finite-perimeter set $\Omega\subset M$ is outward minimizing if
\begin{equation*}
    P(E)\geq P(\Omega)
\end{equation*}
for every bounded finite-perimeter set $E$ that contains $\Omega$ up to measure zero.  It is called strictly outward minimizing if equality implies that $E=\Omega$ up to measure zero.  This is a variational perimeter condition rather than a pointwise curvature condition.  In Euclidean space, smooth star-shaped strictly mean-convex domains are strictly outward minimizing \cite[Corollary~3.8]{FM22}, but the outward-minimizing class is much larger and the classical inverse mean curvature flow may develop singularities from such initial data.  We therefore use the weak inverse mean curvature flow of Huisken--Ilmanen \cite{HI2001}.  

Our first theorem is the Euclidean Weinstock inequality in the outward-minimizing class.

\begin{theorem}\label{thm:weinstock}
Let $\Omega\subset\R^n$, $n\geq3$, be a bounded outward-minimizing domain with smooth  boundary $\Sigma=\partial\Omega$.  Then
\begin{equation}\label{eq:weinstock-main}
    \sigma_1(\Omega)|\partial\Omega|^{\frac1{n-1}}
    \leq
    \sigma_1(B^n)|\partial B^n|^{\frac1{n-1}}
    =|\Sn^{n-1}|^{\frac1{n-1}}.
\end{equation}
Equality holds if and only if $\Omega$ is a round ball.
\end{theorem}

A key ingredient for proving Theorem \ref{thm:weinstock} is a three-term geometric inequality in Proposition \ref{prop:weighted-three-term}.  If $r=|x|$ is the Euclidean distance from a fixed origin, it asserts that for every $k>0$
\begin{equation*}
    \int_\Sigma r^k\,d\mu
    \geq
    \frac{n-1}{n-1+k}|\Sn^{n-1}|^{-\frac{k}{n-1}}
    |\Sigma|^{\frac{n-1+k}{n-1}}
    + k\int_\Omega r^{k-1}\,dx.
\end{equation*}
This extends the Euclidean three-term inequality of Kwong--Wei \cite{KwongWei2023} from the smooth star-shaped mean-convex setting to the outward-minimizing setting. Let $\{\Omega_t\}_{t\geq0}$ denote the sublevel-set domains of the weak
inverse mean curvature flow starting from $\Omega$, and set
\[
    \Sigma_t:=\partial^*\Omega_t,
    \qquad
    \mu_t:=\mathcal H^{n-1}\mathbin{\llcorner}\Sigma_t.
\]
The proof is based on the monotonicity of
\begin{equation*}
    |\Sigma_t|^{-\frac{n-1+k}{n-1}}
    \left(\int_{\Sigma_t}r^k\,d\mu_t
        -k\int_{\Omega_t}r^{k-1}\,dx\right)
\end{equation*}
along the weak flow.

Our second theorem is the hyperbolic analogue.

\begin{theorem}\label{thm:hyp-outward-weinstock}
Let $\Omega\subset\HH^n$, $n\geq3$,  be a bounded outward-minimizing domain with smooth boundary $\Sigma=\partial\Omega$.  Let $\Omega^*$ be a geodesic ball satisfying $ |\partial\Omega^*|=|\partial\Omega|$.  Then
\begin{equation}\label{eq:hyp-weinstock-final}
    \sigma_1(\Omega)\leq\sigma_1(\Omega^*).
\end{equation}
Equality holds if and only if $\Omega$ is a geodesic ball.
\end{theorem}

We now describe the proof and the main new points.  In both the Euclidean and hyperbolic arguments, the monotonicity along smooth inverse mean curvature flow is replaced by endpoint distributional monotonicity along the weak flow.  This use of weak IMCF is in the same spirit as its applications to Minkowski-type inequalities for outward-minimizing hypersurfaces.  In Euclidean space, Huisken \cite{Huisken2009Lecture} used weak IMCF to prove the corresponding Minkowski inequality, see also Freire--Schwartz \cite{FS2014}.  Agostiniani--Fogagnolo--Mazzieri \cite{AFM2022} later gave a nonlinear potential theoretic approach.  In hyperbolic space, Harvie \cite{Harvie2026} proved weak-IMCF Minkowski-type inequalities for outward-minimizing domains in dimensions $3\leq n\leq7$.  The measure argument below applies in every dimension $n\geq3$.

The main difficulty is to pass from smooth inverse mean curvature
flow to weak inverse mean curvature flow without losing the
prescribed initial data. After extending $u$ by zero on $\Omega$,
the weak flow admits, for $t>0$, two natural choices of level-set
domains,
\begin{equation*}
    \Omega_t=\{u<t\},
    \qquad
    \Omega_t^+=\operatorname{int}\{u\leq t\}.
\end{equation*}
These sets may differ at a jump time. In particular, the
right-hand limit at $t=0$ may be the strictly minimizing hull
$\Omega_0^+$ rather than the prescribed domain $\Omega$. This
distinction is harmless in some geometric inequalities, but it
matters here because the Steklov estimates are formulated in terms
of the original boundary $\partial\Omega$.

The key device in this paper is the calibrated formulation of weak IMCF \cite[Section~3]{HI2001}.  The calibration is an $L^\infty$ divergence-measure field.  Applying the Gauss--Green formula to a weighted multiple of this field produces a normal trace on the prescribed initial hypersurface.  The trace has the sign needed for monotonicity, and the contribution from the set $\{u=0\}\setminus\overline\Omega$ is included in a nonnegative defect measure.  Thus the distributional identities are written on $[0,\infty)$ with the endpoint value computed on $\Sigma$, not on $\partial\Omega_0^+$. No strict outward-minimizing hypothesis is used.

In the Euclidean case, the endpoint transport formula gives the defect-measure identity \eqref{eq:F-defect-measure} behind Proposition \ref{prop:monotonicity}.  In the hyperbolic case, it gives a measure inequality \eqref{eq:hyp-measure-transport} for $G(t)=\int_{\Sigma_t}g\,d\mu_t$.  We combine this inequality with a time-mollification argument for the denominator of the Gu--Li--Wan \cite{GuLiWan2025} functional
\begin{equation}\label{s1.Mt}
    \mathcal M(t)=\frac{1}{|\Sigma_t|}
    \frac{\int_{\Sigma_t}g\,d\mu_t}
        {|\Omega_t|h(\int_{\Omega_t}\frac{\lambda'g}{\lambda}\,dv)}.
\end{equation}
The weighted ball integral is convex as a function of the ball volume. Jensen's inequality therefore preserves the mass-transplantation estimate under time mollification, while the
measure formulation automatically includes all jump contributions. This yields monotonicity of $ \mathcal M(t)$ along weak IMCF in hyperbolic space. The remaining hyperbolic estimates are the radial mass-transplantation inequalities of Gu--Li--Wan \cite{GuLiWan2025} and the low-dimensional comparison from Gu's thesis \cite{GuThesis2026}.

The paper is organized as follows.  Section \ref{sec.2} recalls the weak inverse mean curvature flow facts used later and develops the calibrated endpoint transport formula.  Section \ref{sec.3} proves the Euclidean weighted three-term inequality and then Theorem \ref{thm:weinstock}.  Section \ref{sec.4} proves the hyperbolic weak-flow monotonicity and then Theorem \ref{thm:hyp-outward-weinstock}.

\section{Weak inverse mean curvature flow}\label{sec.2}

We recall the weak inverse mean curvature flow in the form needed below.  Huisken--Ilmanen \cite{HI2001} formulated the weak flow on general ambient Riemannian manifolds.  In this paper the ambient manifold $M$ is either $\R^n$ or $\HH^n$.  For a finite-perimeter set $E$ we write $P(E)=\calH^{n-1}(\partial^*E)$. This is the reduced-boundary representation of the perimeter.  See \cite[Theorem~15.9]{Maggi2012}.  All integrals over weak level-set boundaries are taken over the reduced boundary.

Let $\Omega\subset M$ be a bounded domain with smooth boundary $\Sigma$.  A proper function
\begin{equation*}
    u\in C^0(M\setminus\Omega)
      \cap W_{\mathrm{loc}}^{1,\infty}(M\setminus\overline\Omega)
\end{equation*}
is a weak solution of inverse mean curvature flow with initial data $\Sigma$ if its continuous boundary trace is zero on $\Sigma$, if $u\to\infty$ at infinity, and if the sublevel sets
\begin{equation*}
    \Omega_t:=\Omega\cup\{x\in M\setminus\overline\Omega:u(x)<t\},
    \qquad
    \Sigma_t:=\partial^*\Omega_t
\end{equation*}
solve
\begin{equation}\label{eq:level-set-imcf}
    \diver_M\left(\frac{\nabla u}{|\nabla u|}\right)=|\nabla u|
\end{equation}
in the variational sense of Huisken--Ilmanen.  If $u$ is smooth and $|\nabla u|\neq0$, then \eqref{eq:level-set-imcf} is equivalent to the classical inverse mean curvature flow \cite{Gerhardt1990,Gerhardt2011,Urbas1990}
\begin{equation*}
    \partial_t X=\frac{1}{H}\nu.
\end{equation*}

The explicit expanding-sphere solutions in $\R^n$ and $\HH^n$ are proper weak subsolutions.  The weak existence theorem \cite[Theorem~3.1]{HI2001} therefore gives a proper weak solution for every smooth bounded initial domain in the settings used here. The solution obtained by this construction satisfies
$u\geq0$ on $M\setminus\Omega$. Moreover, the boundary gradient estimate in the proof of \cite[Theorem~3.1]{HI2001}, together with the continuous zero trace of $u$ on $\Sigma$, implies that the extension obtained by setting $u=0$ on $\Omega$ belongs to $W_{\mathrm{loc}}^{1,\infty}(M)$. We continue to denote this extension by $u$. For $t>0$, the sets $\Omega_t$ are outward minimizing.  The right-continuous representatives
\begin{equation*}
    \Omega_t^+:=\operatorname{int}\bigl(\Omega\cup\{x\in M\setminus\overline\Omega:u(x)\leq t\}\bigr)
\end{equation*}
are strictly outward minimizing by the minimizing hull property \cite[Minimizing Hull Property~1.4]{HI2001}.  A jump occurs when $\Omega_t$ and $\Omega_t^+$ differ by a set of positive measure.

The monotonicity of the sublevel sets gives
\begin{equation*}
    \chi_{\Omega_s}\longrightarrow\chi_{\Omega_t}
    \quad\text{in }L^1_{\mathrm{loc}}\text{ as }s\nearrow t,
\end{equation*}
and
\begin{equation*}
    \chi_{\Omega_s}\longrightarrow\chi_{\Omega_t^+}
    \quad\text{in }L^1_{\mathrm{loc}}\text{ as }s\searrow t.
\end{equation*}
At $t=0$, the right limit may be the strictly minimizing hull $\Omega_0^+$ rather than the prescribed domain $\Omega$.  We set $\Omega_0=\Omega$ and $\Sigma_0=\Sigma$ throughout.  Every distributional identity below contains an explicit endpoint term at $t=0$ computed on the prescribed initial domain.

If the initial domain is outward minimizing, the exponential area law \cite[Lemma~1.6]{HI2001} gives
\begin{equation}\label{eq:area-growth}
    P(\Omega_t)=e^tP(\Omega),
    \qquad t\geq0.
\end{equation}
The equality at $t=0$ uses the prescribed initial trace.  We write $|\Sigma_t|=P(\Omega_t)$.

The proper weak solution obtained through the elliptic-regularization construction in \cite[Theorem 3.1]{HI2001} carries the corresponding calibration. More precisely, there exists a measurable vector field $\nu$ on $M\setminus\overline\Omega$ such that (see \cite[(3.15)]{HI2001})
\begin{equation}\label{eq:calibration}
    |\nu|\leq1,
    \qquad
    \langle\nabla u,\nu\rangle=|\nabla u|
    \quad\text{a.e. in }M\setminus\overline\Omega,
\end{equation}
and
\begin{equation}\label{eq:calibration-div}
    \int_{M\setminus\overline\Omega}\langle\nabla\phi,\nu\rangle\,dv
    =-\int_{M\setminus\overline\Omega}\phi|\nabla u|\,dv,
    \qquad
    \phi\in C_c^1(M\setminus\overline\Omega).
\end{equation}
Thus $\diver\nu=|\nabla u|$ in distributions.  At regular points of a level set, $\nu=\nabla u/|\nabla u|$ is the outward unit normal.

We use one endpoint convention throughout the paper.  If $\Phi\in L^1_{\mathrm{loc}}([0,\infty))$ has the prescribed value $\Phi(0)$, let $D_t\Phi$ denote the distributional derivative on $\mathbb R$ of its constant extension to $(-\infty,0)$, obtained by setting $\Phi(t)=\Phi(0)$ for $t<0$.  Then every $\eta\in C_c^1(\mathbb R)$ satisfies
\begin{equation}\label{eq:endpoint-derivative-convention}
    \langle D_t\Phi,\eta\rangle
    =-\int_0^\infty\Phi(t)\eta'(t)\,dt-\eta(0)\Phi(0).
\end{equation}
Here and below, $\eta\in C_c^1([0,\infty))$ means that $\eta$ is the restriction to $[0,\infty)$ of some function $\widetilde{\eta}\in C_c^1(\mathbb R)$. In particular, $\eta(0)$ is not required to vanish. If, in addition, the right trace $\Phi(0+)$ exists and $D_t\Phi$ is a Radon measure, then
\begin{equation*}
    D_t\Phi(\{0\})=\Phi(0+)-\Phi(0).
\end{equation*}
Thus the difference between the prescribed endpoint value and the positive-time right trace is explicitly retained at $t=0$.

All distributional identities obtained from this convention are identities on $\mathbb R$.  When a density such as $A(t)\,dt$ is defined only for $t>0$, it means $A(t)\mathbf 1_{(0,\infty)}(t)\,dt$ on $\mathbb R$.

\subsection{An endpoint Gauss--Green formula}\label{subsec:endpoint-GG}

The next lemma is the basic endpoint identity used in both ambient geometries.  The normal trace is understood in the sense of divergence-measure fields.  We use the Gauss--Green formula and the normal trace from \cite[Theorem~2.2]{ChenFrid1999}, and the product rule from \cite[Theorem~3.1]{ChenFrid1999}.  These results are local, and the coordinate reduction below gives the corresponding formula on a smooth
Riemannian manifold.  For a recent intrinsic formulation on regular
domains in metric measure spaces, including the Gauss--Green formula
and the sharp normal-trace estimate, see \cite[Theorem~3.6]{GornyMazon2024}.

\begin{lemma}\label{lem:endpoint-transport}
Let $f\in W^{1,\infty}_{\mathrm{loc}}(M)$ be nonnegative.  For almost every $t>0$ set
\begin{equation*}
    A_f(t):=\int_{\partial^*\Omega_t}f\,d\calH^{n-1},
    \qquad
    A_f(0):=\int_\Sigma f\,d\mu.
\end{equation*}
Then $A_f\in L^1_{\mathrm{loc}}([0,\infty))$.  There is a normal trace
\begin{equation*}
    \tau:=\operatorname{Tr}_\Sigma(\nu\cdot\nu_\Omega),
    \qquad |\tau|\leq1,
\end{equation*}
such that every $\eta\in C_c^1([0,\infty))$ satisfies
\begin{align}
    -\int_0^\infty\eta'(t)A_f(t)\,dt-\eta(0)A_f(0)
    =&\int_0^\infty\eta(t)A_f(t)\,dt 
      +\int_{M\setminus\overline\Omega}\eta(u)\langle\nabla f,\nu\rangle\,dv \notag\\
    &+\eta(0)\int_\Sigma f(\tau-1)\,d\mu.\label{eq:endpoint-transport-exact}
\end{align}
Consequently, if $\eta\geq0$, then
\begin{equation}\label{eq:endpoint-transport-ineq}
    -\int_0^\infty\eta'(t)A_f(t)\,dt-\eta(0)A_f(0)
    \leq
    \int_0^\infty\eta(t)A_f(t)\,dt
    +\int_{M\setminus\overline\Omega}\eta(u)\langle\nabla f,\nu\rangle\,dv.
\end{equation}
\end{lemma}

\begin{proof}
Since $u$ is proper and $\eta$ has compact support, the vector field
\begin{equation*}
    X:=f\eta(u)\nu
\end{equation*}
has compact support in the closure of $M\setminus\overline\Omega$.  Equations \eqref{eq:calibration} and \eqref{eq:calibration-div}, together with the product rule for divergence-measure fields, give
\begin{equation}
    \diver X
    =\eta(u)\langle\nabla f,\nu\rangle
     +f\bigl(\eta'(u)+\eta(u)\bigr)|\nabla u|
     \label{equ-divX}
\end{equation}
as an identity of distributions on $M\setminus\overline\Omega$.  In particular, $\diver X$ is represented there by a locally integrable function.

We next explain why the Gauss--Green formula applies in the Riemannian
setting. Choose a bounded smooth open set $U$ containing $\overline\Omega$ and the support of $X$, with $X=0$ near $\partial U$.  In a coordinate chart $y=(y^1,\ldots,y^n)$, write $X=X^i\partial_i$ and define the Euclidean coordinate field
\begin{equation*}
    Y^i:=\sqrt{\det(g_{ab})}\,X^i.
\end{equation*}
Then
\begin{equation*}
    \diver_{\mathbb R^n}Y\,dy
    =\diver_gX\,dv_g.
\end{equation*}
Thus the Euclidean divergence-measure theory in
\cite[Theorems~2.2 and~3.1]{ChenFrid1999} applies to $Y$ in every chart.  Under this coordinate change, its boundary flux becomes the intrinsic Riemannian flux $\langle X,\nu_E\rangle_g\,d\calH_g^{n-1}$, where $\nu_E$ is the measure-theoretic outward unit normal of the set under consideration.  A finite coordinate cover and a partition of unity therefore give the Gauss--Green formula on $U\setminus\overline\Omega$.

The same coordinate argument applied to the calibration field $\nu$ gives its normal trace on $\Sigma$.  This does not assert that $\nu$ has a pointwise boundary value. It means that the flux of $\nu$ through $\Sigma$, taken from the side $M\setminus\overline\Omega$, is well defined by the Gauss--Green formula.  With $\nu_{\Omega}$ denoting the unit normal pointing out of $\Omega$, write
\begin{equation*}
    \tau=\operatorname{Tr}_\Sigma(\nu\cdot\nu_\Omega).
\end{equation*}

Using the normal exponential deformation of $\Sigma$ in the weak-star characterization of the normal
trace in \cite[Theorem~2.2(iii)]{ChenFrid1999} and the bound
$|\nu|\leq1$, we obtain
\begin{equation}\label{eq:normal-trace-bound}
    \|\tau\|_{L^\infty(\Sigma)}
    \leq
    \|\nu\|_{L^\infty(U\setminus\overline\Omega)}
    \leq1.
\end{equation}
Since $\eta(u)$ is locally Lipschitz and has trace
$\eta(0)$ on $\Sigma$, the product rule and the characterization of the normal trace therefore give
\begin{equation*}
    \operatorname{Tr}_\Sigma(X\cdot\nu_\Omega)
    =
    f\eta(0)\operatorname{Tr}_\Sigma(\nu\cdot\nu_\Omega)
    =
    f\eta(0)\tau.
\end{equation*} 
The flux across $\partial U$ is zero.  Since the outer unit normal of $U\setminus\overline\Omega$ along $\Sigma$ is $-\nu_\Omega$, the Gauss--Green formula yields
\begin{equation}\label{s2.pf1}
    \int_{M\setminus\overline\Omega}\diver X\,dv
    =-\eta(0)\int_\Sigma f\tau\,d\mu.
\end{equation}

For each $T>0$, properness of $u$ implies that $\{0\leq u\leq T\}$ is compact. Since $f$ and $|\nabla u|$ are locally bounded, the coarea formula gives
\begin{equation*}
    \int_0^T A_f(t) \dd t = \int_{\{0<u<T\}} f |\nabla u| \dd v<\infty.
\end{equation*}
Consequently $A_f\in L_{\mathrm{loc}}^1([0,\infty))$. Since $u\in W_{\mathrm{loc}}^{1,\infty}(M)$ and
$\Omega_t=\{u<t\}$ for every $t>0$, applying the coarea formula to \eqref{equ-divX} gives
\begin{align}\label{s2.pf2}
 \int_{M\setminus\overline\Omega}\diver X\,dv
   =&  \int_{M\setminus\overline\Omega}\eta(u)\langle\nabla f,\nu\rangle \,dv+
    \int_{M\setminus\overline\Omega}
        f\bigl(\eta'(u)+\eta(u)\bigr)|\nabla u|\,dv\nonumber\\
    =&\int_{M\setminus\overline\Omega}\eta(u)\langle\nabla f,\nu\rangle \,dv+
    \int_0^\infty
        (\eta'(t)+\eta(t))
        \left(
            \int_{\partial^*\Omega_t}
                f\,d\mathcal H^{n-1}
        \right)dt \nonumber\\
    =&\int_{M\setminus\overline\Omega}\eta(u)\langle\nabla f,\nu\rangle \,dv+
    \int_0^\infty
        (\eta'(t)+\eta(t))A_f(t)\,dt.
\end{align}
Properness of $u$ and the local boundedness of $f$ ensure that all the integrals are finite. The set $\{u=0\}\setminus\overline\Omega$ causes no additional coarea term, because $|\nabla u|=0$ almost everywhere on this set. The coarea formula determines $A_f(t)$ only for almost every $t>0$. It does not determine the separately prescribed endpoint value $A_f(0)=\int_\Sigma f\,d\mu$.

Combining \eqref{s2.pf1} and \eqref{s2.pf2} and rearranging gives \eqref{eq:endpoint-transport-exact}.  Since $f\geq0$ and $\tau\leq1$ by \eqref{eq:normal-trace-bound}, the boundary defect term is non-positive
for every nonnegative $\eta$.  Hence
\eqref{eq:endpoint-transport-ineq} follows.
\end{proof}

\section{The Euclidean case}\label{sec.3}

In this section, we prove the Euclidean Weinstock inequality.  The first step is a weak-flow proof of the weighted three-term inequality.  We apply the endpoint Gauss--Green formula from Lemma \ref{lem:endpoint-transport}, combine it with a finite-perimeter divergence estimate, and obtain monotonicity of a normalized quantity. The final subsection applies the case $k=1$ to the Steklov variational characterization.

Throughout this section, we fix an origin and let $r(x):=|x|$ denote the Euclidean distance from it. Away from the origin, we write $\partial_r:=\nabla r=\frac{x}{r}$. Whenever $\partial_r$ occurs in an almost-everywhere identity, its value at the origin may be chosen arbitrarily. We reserve $D$ for
distributional derivatives.

\subsection{The weighted three-term inequality}

\begin{proposition}\label{prop:weighted-three-term}
Let $\Om\subset\R^n$, $n\geq3$, be a bounded outward-minimizing domain with smooth boundary $\Sig=\partial\Om$.  Fix an origin in $\R^n$ and set $r=|x|$.  Then for every $k>0$,
\begin{equation}\label{eq:weighted-main}
    \int_\Sig r^k\,d\mu
    \geq
    \frac{n-1}{n-1+k}|\Sn^{n-1}|^{-\frac{k}{n-1}}
    |\Sig|^{\frac{n-1+k}{n-1}}
    + k\int_\Om r^{k-1}\,dx.
\end{equation}
Equality holds if and only if $\Sig$ is a round sphere centered at the chosen origin.
\end{proposition}

Fix $k>0$. For every $t>0$, set
\begin{equation*}
    B_k(t):=k\int_{\Om_t}r^{k-1}\,dx,
\end{equation*}
and, for almost every $t>0$, set
\begin{equation*}
    A_k(t):=\int_{\Sig_t}r^k\,d\mu_t,
    \qquad
    F_k(t):=A_k(t)-B_k(t).
\end{equation*}
At $t=0$, prescribe
\begin{equation*}
    A_k(0):=\int_{\Sig}r^k\,d\mu,
    \qquad
    B_k(0):=k\int_{\Om}r^{k-1}\,dx,
    \qquad
    F_k(0):=A_k(0)-B_k(0).
\end{equation*}
Whenever distributional derivatives in time are used, these functions are extended to $t<0$ by their prescribed values at $t=0$.

\begin{lemma}\label{lem:transport}
Let $\{\Om_t\}$ be the weak inverse mean curvature flow starting from a smooth bounded domain $\Om$.  Then $F_k\in L^1_{\mathrm{loc}}([0,\infty))$ and
\begin{equation*}
    D_tF_k\leq A_k\,dt
\end{equation*}
as an inequality of distributions on $\mathbb R$ under the endpoint convention above. Here $A_k\,dt$ is zero on $(-\infty,0)$.  Equivalently, for every nonnegative $\eta\in C_c^1([0,\infty))$,
\begin{equation}\label{eq:distribution-F}
    -\int_0^\infty\eta'(t)F_k(t)\,dt-\eta(0)F_k(0)
    \leq
    \int_0^\infty\eta(t)A_k(t)\,dt.
\end{equation}
More precisely, define
\begin{equation}\label{eq:Euclidean-defect-measure-definition}
\begin{split}
    \mathfrak e_k
    :=&{u}_{\#}\left(
      kr^{k-1}\bigl(1-\partial_r\cdot\nu\bigr)
      \mathcal L^n\mathbin{\llcorner}
      (\R^n\setminus\overline\Om)\right)\\
    &+
      \left(\int_\Sig r^k(1-\tau)\,d\mu\right)\delta_0,
\end{split}
\end{equation}
where $\tau$ is the normal trace from Lemma \ref{lem:endpoint-transport}.  This is a nonnegative Radon measure and
\begin{equation}\label{eq:F-defect-measure}
    D_tF_k=A_k\,dt-\mathfrak e_k.
\end{equation}
Equivalently, for every $\eta\in C_c([0,\infty))$,
\begin{align}
    \int_{[0,\infty)}\eta\,d\mathfrak e_k
    &=\int_{\R^n\setminus\overline\Om}
      \eta(u)kr^{k-1}\bigl(1-\partial_r\cdot\nu\bigr)\,dx \notag\\
    &\quad+\eta(0)\int_\Sig r^k(1-\tau)\,d\mu.\label{eq:Euclidean-defect}
\end{align}
\end{lemma}

\begin{proof}
For $\delta>0$ define
\begin{equation*}
    f_\delta(x):=(r^2+\delta^2)^{k/2}.
\end{equation*}
The regularization removes the possible singularity of $\nabla r^k$ at the origin and allows a single argument for every $k>0$.  The function $f_\delta$ is smooth and nonnegative, so Lemma \ref{lem:endpoint-transport} applies to it.  We now justify the passage to the limit $\delta\downarrow0$.  Fix $\eta$ and choose $T>0$ such that the supports of $\eta$ and $\eta'$ are contained in $[0,T]$.  By the properness of $u$, the set
\begin{equation*}
    K_T:=\overline{\{x\in\R^n\setminus\overline\Om:0\leq u(x)\leq T\}}
\end{equation*}
is compact.  For $\zeta=\eta$ and $\zeta=\eta'$, the coarea formula gives
\begin{equation*}
    \int_0^\infty \zeta(t)A_{f_\delta}(t)\,dt
    =\int_{\R^n\setminus\overline\Om}
      \zeta(u)f_\delta|\nabla u|\,dx.
\end{equation*}
Since $f_\delta\to r^k$ uniformly on $K_T$ and $|\nabla u|$ is locally integrable, these two terms converge to the corresponding terms with $r^k$.  The endpoint term also converges because $f_\delta\to r^k$ uniformly on the smooth compact hypersurface $\Sig$.

It remains to pass to the limit in the gradient term.  On $K_T$,
\begin{equation*}
    \nabla f_\delta
    =kr(r^2+\delta^2)^{\frac{k}{2}-1}\partial_r
    \longrightarrow kr^{k-1}\partial_r
\end{equation*}
almost everywhere.  If $0<k<2$, then $|\nabla f_\delta|\leq kr^{k-1}$ away from $\{0\}$, and $r^{k-1}$ is locally integrable for every $k>0$.  If $k\geq2$, then $|\nabla f_\delta|$ is uniformly bounded on $K_T$.  Since $|\nu|\leq1$, dominated convergence gives
\begin{equation*}
    \int_{\R^n\setminus\overline\Om}
      \eta(u)\langle\nabla f_\delta,\nu\rangle\,dx
    \longrightarrow
    \int_{\R^n\setminus\overline\Om}
      \eta(u)kr^{k-1}\partial_r\cdot\nu\,dx.
\end{equation*}
Thus we may let $\delta\downarrow0$ in \eqref{eq:endpoint-transport-exact}. This gives
\begin{align}
    -\int_0^\infty\eta'(t)A_k(t)\,dt-\eta(0)A_k(0)
    &=\int_0^\infty\eta(t)A_k(t)\,dt 
      +\int_{\R^n\setminus\overline\Om}
       \eta(u)kr^{k-1}\partial_r\cdot\nu\,dx \notag\\
    &\quad+\eta(0)\int_\Sig r^k(\tau-1)\,d\mu.\label{eq:Ak-endpoint}
\end{align}

Since $u$ is proper and $kr(x)^{k-1}$ is locally integrable,
\begin{equation*}
    \int_0^\infty\int_{\mathbb R^n\setminus\Omega}
    |\eta'(t)|\mathbf 1_{\{u(x)<t\}}kr(x)^{k-1}\,dx\,dt
    \leq
    T\|\eta'\|_{L^\infty}\int_{K_T}kr(x)^{k-1}\,dx<\infty.
\end{equation*}
Hence Fubini's theorem applies. Since $\eta$ has compact support and
\begin{equation*}
    B_k(t)=B_k(0)
      +\int_{\R^n\setminus\overline\Om}
        \mathbf 1_{\{u<t\}}kr^{k-1}\,dx,
\end{equation*}
we have 
\begin{align}\label{eq:Bk-distribution}
 -\int_0^\infty\eta'(t)&B_k(t)\,dt-\eta(0)B_k(0)\nonumber\\
 =&-\int_0^\infty\eta'(t)\left(B_k(0)
      +\int_{\R^n\setminus\overline\Om}
        \mathbf 1_{\{u<t\}}kr^{k-1}\,dx\right)\,dt-\eta(0)B_k(0)\nonumber\\
=&-B_k(0)\int_0^\infty\eta'(t)dt-\eta(0)B_k(0)\nonumber\\
&\qquad -\int_0^\infty\eta'(t)\int_{\R^n\setminus\overline\Om}
        \mathbf 1_{\{u<t\}}kr^{k-1}\,dx \,dt\nonumber\\
 =&-\int_{\mathbb R^n\setminus\Omega}kr^{k-1}
   \int_0^\infty\eta'(t)\mathbf 1_{\{u(x)<t\}}\,dt\,dx\nonumber\\
 =&-\int_{\mathbb R^n\setminus\Omega}kr^{k-1}
   \int_{u(x)}^\infty\eta'(t)\,dt\,dx\nonumber\\
 =&\int_{\mathbb R^n\setminus\Omega}
   \eta(u)kr^{k-1}\,dx.
\end{align}
Subtracting \eqref{eq:Bk-distribution} from \eqref{eq:Ak-endpoint} gives \eqref{eq:F-defect-measure} with the measure in \eqref{eq:Euclidean-defect-measure-definition}.   The identity first holds for test functions in $C_c^1([0,\infty))$ and then for every function in $C_c([0,\infty))$ by density.  The measure is nonnegative because $|\nu|\leq1$ and $\tau\leq1$.  Equation \eqref{eq:distribution-F} follows.
\end{proof}

We next relate $A_k(t)$ and $B_k(t)$ by a divergence estimate.

\begin{lemma}\label{lem:divergence}
Let $E\subset\R^n$ be a finite-perimeter set with compact closure.  For every $k>0$,
\begin{equation}\label{eq:divergence-estimate}
    (n-1+k)\int_E r^{k-1}\,dx
    \leq
    \int_{\partial^*E}r^k\,d\calH^{n-1}.
\end{equation}
If $|E|>0$, equality holds if and only if $E$ agrees almost everywhere with a ball centered at the chosen origin. In particular, for almost every $t>0$ along the weak flow,
\begin{equation}\label{eq:B-vs-A}
    F_k(t)\geq\frac{n-1}{n-1+k}A_k(t).
\end{equation}
\end{lemma}

\begin{proof}
Consider the smooth vector fields
\begin{equation*}
    X_\delta=(r^2+\delta^2)^{\frac{k-1}{2}}x.
\end{equation*}
The limiting field $r^{k-1}x$ has the desired divergence but may fail to be $C^1$ at the origin when $0<k<1$.  The field $X_\delta$ regularizes this singularity.  Since $E$ has compact closure, we may multiply $X_\delta$ by a smooth cutoff that equals one near $\overline E$.  The Gauss--Green theorem for finite-perimeter sets \cite[Theorem~15.9]{Maggi2012} then yields
\begin{equation}\label{s3.delta}
\int_E\diver X_\delta\,dx
=\int_{\partial^*E} X_\delta\cdot\nu_E\,d\calH^{n-1}.
\end{equation}
A direct computation gives
\begin{equation*}
    \diver X_{\delta} = n \left( r^2 + \delta^2 \right)^{\frac{k-1}{2}} + (k-1)\left( r^2 + \delta^2 \right)^{\frac{k-3}{2}}r^2.
\end{equation*}
Passing to the limit as $\delta\downarrow 0$ in \eqref{s3.delta} gives
\begin{equation}\label{eq:weighted-divergence-identity}
    (n-1+k)\int_E r^{k-1}\,dx
    =\int_{\partial^*E}r^{k-1}x\cdot\nu_E\,d\calH^{n-1}.
\end{equation}
Indeed, if $0<k<1$, then
\begin{equation*}
    |\diver X_\delta|  \leq (n+1) r^{k-1},
    \qquad
    |X_\delta|\leq r^k,
\end{equation*}
and $r^{k-1}$ is locally integrable in $\mathbb{R}^n$, while for $k\geq1$ both quantities are locally bounded uniformly in $\delta$.  Dominated convergence therefore applies to the volume term and the reduced-boundary term in \eqref{s3.delta}. Since $x\cdot\nu_E\leq r$, inequality \eqref{eq:divergence-estimate} follows.

Suppose equality holds and $|E|>0$.  Equation \eqref{eq:weighted-divergence-identity} implies
\begin{equation*}
    \nu_E=\partial_r
    \quad\text{for }\calH^{n-1}\text{-almost every point of }\partial^*E.
\end{equation*}
Indeed, the nonnegative integrand $r^{k-1}(r-x\cdot\nu_E)$ has zero integral. Recall that
\begin{equation*}
    D\chi_E=-\nu_E|D\chi_E|.
\end{equation*}
Then we have $D\chi_E=-\partial_r|D\chi_E|$, where $|D\chi_E|=\mathcal{H}^{n-1}\mathbin{\llcorner}\partial^*E$.  The value of $\partial_r$ at the origin is immaterial. We first prove that $E$ is rotationally invariant. Let $A$ be a skew-symmetric matrix and set $Z_A(x)=Ax$.  Since 
\begin{equation*}
    Z_A\cdot \partial_r=\frac{1}{r}Ax\cdot x=\frac{1}{r}x^TA^Tx=0
\end{equation*}
and $\diver Z_A=\operatorname{tr}A=0$,
\begin{equation*}
    \diver(\chi_E Z_A)
    =Z_A\cdot D\chi_E+\chi_E\diver Z_A=0
\end{equation*}
in the sense of distributions. If $R_t=e^{tA}$, then for every $\varphi\in C_c^1(\R^n)$,
\begin{align*}
\frac{d}{dt}\int_{\mathbb R^n}
 \chi_E(x)\varphi(R_{-t}x)\,dx
&=-\int_{\mathbb R^n}
 \chi_E(x)Z_A(x)\cdot
 \nabla(\varphi\circ R_{-t})(x)\,dx\\
&=\left\langle
 \diver(\chi_EZ_A),\varphi\circ R_{-t}
 \right\rangle=0.
\end{align*}
Thus the integral is independent of $t$. By a change of variables $x=R_ty$ and noting that $\det R_t=1$, 
\begin{equation*}
    \int_{\R^n}\chi_E(R_ty)\varphi(y)dy=\int_{\R^n}\chi_E(y)\varphi(y)dy
\end{equation*}
for every $\varphi\in C_c^1(\R^n)$. This yields $\chi_E\circ R_t=\chi_E$ almost everywhere. Since $A$ was arbitrary and the rotations $e^{tA}$ generate $\mathrm{SO}(n)$, the function $\chi_E$ is invariant almost everywhere under every rotation about the origin. A standard averaging argument
over $\mathrm{SO}(n)$, together with Fubini's theorem, then shows that $\chi_E$ agrees almost everywhere with a radial function. Since $\chi_E$ takes only the values zero and one, there exists a measurable function $\gamma:(0,\infty)\to\{0,1\}$ such that $\chi_E(x)=\gamma(|x|)$ almost everywhere.

To determine the radial profile, let $0\leq\psi\in C_c^1((0,\infty))$ and define
\begin{equation*}
    Y(x)=\frac{\psi(r)}{r^{n-1}}\partial_r.
\end{equation*}
Then $Y$ is smooth and compactly supported away from the origin, and $\diver Y=r^{1-n}\psi'(r)$.  The Gauss--Green formula and $\nu_E=\partial_r$ give
\begin{align*}
    |\Sn^{n-1}|\int_0^\infty \gamma(r)\psi'(r)\,dr
    &=\int_E\diver Y\,dx\\
    &=\int_{\partial^*E}\frac{\psi(r)}{r^{n-1}}\,d\calH^{n-1}\geq0.
\end{align*}
Thus $D_r\gamma\leq0$ in distributions, so $\gamma$ has a non-increasing representative.  Since $\gamma$ takes only the values zero and one, compactness and positive measure imply
\begin{equation*}
    \chi_E=\mathbf 1_{B_R}
\end{equation*}
almost everywhere for some $R>0$.  The converse follows by direct computation.  Applying \eqref{eq:divergence-estimate} to $E=\Om_t$ gives \eqref{eq:B-vs-A} for almost every $t>0$.
\end{proof}

\begin{proposition}[Monotonicity]\label{prop:monotonicity}
Let $\Om\subset\R^n$ be a smooth outward-minimizing domain, and let $\{\Om_t\}$ be the weak inverse mean curvature flow starting from $\Om$. For every $k>0$, define, for almost every $t>0$ and at the prescribed value $t=0$,
\begin{equation*}
    Q_k(t):=|\Sig_t|^{-\frac{n-1+k}{n-1}}F_k(t).
\end{equation*}
Then $Q_k$ agrees almost everywhere on $(0,\infty)$ with a non-increasing function, whose value at $t=0$ is the prescribed value computed on the initial boundary.
\end{proposition}

\begin{proof}
First define $H_k(t):=e^{-\frac{n-1+k}{n-1} t}F_k(t)$ for $t\geq0$, with $H_k(0)=F_k(0)$, and then extend $H_k$ itself constantly to $t<0$. Equivalently, $H_k=wF_k$ on $\mathbb R$, where $w(t)=1$ for $t<0$ and $w(t)=e^{-\frac{n-1+k}{n-1} t}$ for $t\geq0$. The product rule for BV functions, \eqref{eq:F-defect-measure}, and \eqref{eq:B-vs-A} give
\begin{equation}\label{eq:Q-defect}
    D_tH_k
    =-e^{-\frac{n-1+k}{n-1} t}\left(
      \mathfrak e_k+\left(\frac{n-1+k}{n-1} F_k-A_k\right)\,dt\right).
\end{equation}
The right-hand side is a nonpositive Radon measure. Hence $H_k$ agrees almost everywhere with a non-increasing function. Since $|\Sig_t|=e^t|\Sig|$ by \eqref{eq:area-growth},
\begin{equation*}
    Q_k(t)=|\Sig|^{-\frac{n-1+k}{n-1}}H_k(t)
\end{equation*}
for almost every $t>0$ and at $t=0$.  This gives the asserted representative of $Q_k$.
\end{proof}

\begin{proof}[Proof of Proposition \ref{prop:weighted-three-term}]

We first compute the limit of $Q_k(t)$ as $t\to\infty$. By the eventual smoothness theorem of Huisken--Ilmanen \cite[Theorem~2.7]{HI2008}, after increasing $T>0$ if necessary, the weak flow agrees for $t\geq T$ with a smooth star-shaped mean-convex inverse mean curvature flow. Set
\begin{equation*}
    \widetilde\Omega_t:=e^{-\frac{t}{n-1}}\Omega_t,
    \qquad
    \widetilde\Sigma_t:=e^{-\frac{t}{n-1}}\Sigma_t .
\end{equation*}
The asymptotic convergence of the rescaled inverse mean curvature flow \cite{Gerhardt1990,Urbas1990} implies that, for all sufficiently large $t$, one may write
\begin{equation*}
    \widetilde\Sigma_t
    =\{\rho_t(\theta)\theta:\theta\in\mathbb S^{n-1}\},
    \qquad
    \rho_t\longrightarrow R_\infty
    \quad\text{in }C^1(\mathbb S^{n-1}),
\end{equation*}
where, by the area law \eqref{eq:area-growth},
\begin{equation*}
    R_\infty
    =\left(\frac{|\Sigma|}{|\mathbb S^{n-1}|}\right)^{\frac1{n-1}}.
\end{equation*}
For every $k>0$, polar coordinates give
\begin{equation*}
    k\int_{\widetilde\Omega_t}r^{k-1}\,dx
    =\frac{k}{n-1+k}
      \int_{\mathbb S^{n-1}}\rho_t^{\,n-1+k}\,d\theta,
\end{equation*}
while the radial graph representation gives
\begin{equation*}
    \int_{\widetilde\Sigma_t}r^k\,d\widetilde\mu_t
    =
    \int_{\mathbb S^{n-1}}
       \rho_t^{\,n-1+k}
       \sqrt{1+\frac{|\overline{\nabla}\rho_t|_{g_{\mathbb{S}^{n-1}}}^2}{\rho_t^2}}\,d\theta,
\end{equation*}
where $\overline{\nabla}$ denotes the Levi-Civita connection on $\mathbb{S}^{n-1}$. Consequently,
\begin{equation*}
    \lim_{t\to\infty}
    \left(
      \int_{\widetilde\Sigma_t}r^k\,d\widetilde\mu_t
      -k\int_{\widetilde\Omega_t}r^{k-1}\,dx
    \right)
    =
    \frac{n-1}{n-1+k}
    |\mathbb S^{n-1}|R_\infty^{\,n-1+k}.
\end{equation*}
Since $Q_k$ is invariant under Euclidean dilations and
\begin{equation*}
    |\widetilde\Sigma_t|
    =|\Sigma|
    =|\mathbb S^{n-1}|R_\infty^{\,n-1},
\end{equation*}
we obtain
\begin{equation}\label{eq:limit-Q}
    \lim_{t\to\infty}Q_k(t)
    =
    \frac{n-1}{n-1+k}
    |\mathbb S^{n-1}|^{-\frac{k}{n-1}}.
\end{equation}
This argument applies to every $k>0$. In particular, the possible singularity of $r^{k-1}$ at the origin when $0<k<1$ is harmless in the polar-coordinate formula above.

By Proposition \ref{prop:monotonicity} and \eqref{eq:limit-Q},
\begin{equation*}
    |\Sig|^{-\frac{n-1+k}{n-1}}
    \left(\int_\Sig r^k\,d\mu-k\int_\Om r^{k-1}\,dx\right)
    =Q_k(0)
    \geq
    \frac{n-1}{n-1+k}|\Sn^{n-1}|^{-\frac{k}{n-1}}.
\end{equation*}
Rearranging gives \eqref{eq:weighted-main}.

Suppose that equality holds, and let $Q_k$ be the non-increasing
representative provided by Proposition \ref{prop:monotonicity}. The
equality assumption and \eqref{eq:limit-Q} give $Q_k(0)=\lim_{s\to\infty}Q_k(s)$.
Therefore, for every $t\geq0$,
\begin{equation*}
    Q_k(0)\geq Q_k(t)\geq
    \lim_{s\to\infty}Q_k(s)=Q_k(0).
\end{equation*}
Hence $Q_k$ is constant and $D_tQ_k=0$ in distributions. Equation \eqref{eq:Q-defect} then implies
\begin{equation}\label{eq:equality-defects-zero}
    \mathfrak e_k=0,
    \qquad
    A_k=\frac{n-1+k}{n-1} F_k
    \quad\text{for almost every }t>0.
\end{equation}

We first show that there is no initial jump.  Assume that $  |\Om_0^+\setminus\Om|>0$.  Since $\partial\Om$ has zero $n$-dimensional measure, the open set $ \mathcal P:=\Om_0^+\setminus\overline\Om$ is nonempty.  The identity $\mathfrak e_k(\{0\})=0$ and the nonnegativity of both terms in \eqref{eq:Euclidean-defect} give
\begin{equation*}
    \partial_r\cdot\nu=1
    \quad\text{almost everywhere in }\mathcal P\setminus\{0\}.
\end{equation*}
Hence $\nu=\partial_r$ there.  Choose a ball $U$ compactly contained in $\mathcal P$ whose closure does not contain the origin.  Since $u$ is constant on $U$, equations \eqref{eq:calibration-div} and \eqref{eq:calibration} give
\begin{equation*}
    \diver\nu=0
    \quad\text{in }U.
\end{equation*}
On the other hand,
\begin{equation*}
    \diver \partial_r=\frac{n-1}{r}>0
    \quad\text{in }U,
\end{equation*}
which is a contradiction.  Thus $\Om_0^+=\Om$ up to a null set.

The second identity in \eqref{eq:equality-defects-zero} is equivalent to equality in \eqref{eq:divergence-estimate} for almost every $t$.  Lemma \ref{lem:divergence} shows that $\Om_t$ is a ball centered at the chosen origin for almost every $t$.  Choose such times $t_j\downarrow0$.  The absence of an initial jump and the right continuity of the weak flow give
\begin{equation*}
    \chi_{\Om_{t_j}}\longrightarrow\chi_\Om
    \quad\text{in }L^1_{\mathrm{loc}}.
\end{equation*}
If $R_j$ is the radius of $\Om_{t_j}$, then the area law \eqref{eq:area-growth} gives
\begin{equation*}
    R_j
    =\left(\frac{e^{t_j}|\Sig|}{|\Sn^{n-1}|}\right)^{\frac1{n-1}}
    \longrightarrow
    \left(\frac{|\Sig|}{|\Sn^{n-1}|}\right)^{\frac1{n-1}}.
\end{equation*}
Together with the local $L^1$ convergence, this shows that $\Om$ agrees almost everywhere with a ball centered at the chosen origin.  Since $\Om$ is a smooth domain, it is that centered ball.

Conversely, equality for a centered ball follows by direct computation.
\end{proof}

\subsection{Proof of the Euclidean Weinstock inequality}

We first record the consequence of the case $k=1$.

\begin{lemma}\label{lem:moment-lower}
Under the assumptions of Proposition \ref{prop:weighted-three-term},
\begin{equation*}
    \int_\Sig r^2\,d\mu
    \geq
    |B^n|^{-\frac2n}|\Sig|\,|\Om|^{\frac2n}.
\end{equation*}
Equality holds if and only if $\Om$ is a round ball centered at the origin.
\end{lemma}

\begin{proof}
The case $k=1$ of \eqref{eq:weighted-main} gives
\begin{equation}\label{eq:k1-three-term}
    \int_\Sig r\,d\mu
    \geq
    \frac{n-1}{n}|\Sn^{n-1}|^{-\frac1{n-1}}|\Sig|^{\frac n{n-1}}
    +|\Om|.
\end{equation}
By H\"older's inequality,
\begin{equation}\label{eq:holder-r}
    \int_\Sig r^2\,d\mu
    \geq
    \frac1{|\Sig|}\left(\int_\Sig r\,d\mu\right)^2.
\end{equation}
Using $|\Sn^{n-1}|=n|B^n|$, Young's inequality gives
\begin{equation*}
    |B^n|^{-1/n}|\Sig|\,|\Om|^{1/n}
    \leq
    |\Om|+\frac{n-1}{n}|\Sn^{n-1}|^{-\frac1{n-1}}|\Sig|^{\frac n{n-1}}.
\end{equation*}
Combining this with \eqref{eq:k1-three-term} gives
\begin{equation*}
    \int_\Sig r\,d\mu
    \geq
    |B^n|^{-1/n}|\Sig|\,|\Om|^{1/n}.
\end{equation*}
Equation \eqref{eq:holder-r} now yields the stated estimate.

Equality forces equality in Proposition \ref{prop:weighted-three-term}, H\"older's inequality, and Young's inequality.  Proposition \ref{prop:weighted-three-term} already implies that $\Om$ is a ball centered at the origin.  The converse is immediate.
\end{proof}

\begin{proof}[Proof of Theorem \ref{thm:weinstock}]
Choose the origin to be the boundary barycenter, so that
\begin{equation*}
    \int_\Sig x_i\,d\mu=0,
    \qquad i=1,\ldots,n.
\end{equation*}
Using $x_i$ as test functions in \eqref{eq:steklov-variational} gives
\begin{equation*}
    \sigma_1(\Om)\int_\Sig x_i^2\,d\mu
    \leq
    \int_\Om|\nabla x_i|^2\,dx
    =|\Om|.
\end{equation*}
Summing over $i$ gives
\begin{equation*}
    \sigma_1(\Om)\int_\Sig r^2\,d\mu
    \leq n|\Om|.
\end{equation*}
By Lemma \ref{lem:moment-lower},
\begin{equation*}
    \sigma_1(\Om)
    \leq
    n|B^n|^{\frac2n}|\Om|^{\frac{n-2}{n}}|\Sig|^{-1}.
\end{equation*}
Multiplying by $|\Sig|^{1/(n-1)}$ and using the isoperimetric inequality
\begin{equation*}
    |\Sig|\geq n|B^n|^{\frac1n}|\Om|^{\frac{n-1}{n}},
\end{equation*}
we obtain
\begin{equation*}
    \sigma_1(\Om)|\Sig|^{\frac1{n-1}}
    \leq
    (n|B^n|)^{\frac1{n-1}}
    =|\Sn^{n-1}|^{\frac1{n-1}}.
\end{equation*}
This is \eqref{eq:weinstock-main}, since $\sigma_1(B^n)=1$ for the unit ball.

If equality holds, then equality holds in Lemma \ref{lem:moment-lower}.  Therefore $\Om$ is a round ball.  Conversely, round balls attain equality.
\end{proof}

\begin{remark}\label{rem:Euclidean-endpoint}
Strict outward minimality is not required. Indeed, the minimizing-hull property and outward minimality give
\[
    P(\Omega_0^+)=P(\Omega)=|\Sigma|,
\]
so the area law remains $|\Sigma_t|=e^t|\Sigma|$. Moreover, \eqref{eq:Q-defect} gives
\begin{align*}
     (D_t Q_k)(\{0\})=&Q_k(0+)-Q_k(0)\\
     =&-|\Sigma|^{-\frac{n-1+k}{n-1}}\mathfrak e_k(\{0\}) \\
    =&-|\Sigma|^{-\frac{n-1+k}{n-1}} \left(
      \int_{\{u=0\}\setminus\overline\Omega}
        kr^{k-1}(1-\partial_r\cdot\nu)\,dx
      +\int_\Sigma r^k(1-\tau)\,d\mu
      \right) \leq 0.
\end{align*}
Here both integrals are nonnegative because $\partial_r\cdot\nu\leq1$ and $\tau\leq1$. Thus a possible enlargement at $t=0$ can only produce a downward jump of $Q_k$ and does not affect its monotonicity.
\end{remark}

\section{The hyperbolic case}\label{sec.4}

We now consider the hyperbolic case.  We first recall the radial functions $g$ and $h$ from Gu--Li--Wan \cite{GuLiWan2025}.  We then prove an endpoint weak-transport inequality for $\int_{\Sigma_t}g\dd\mu_t$.  To handle the denominator of the Gu--Li--Wan functional \eqref{s1.Mt}, we mollify the relevant quantities in time.  The weighted volume integral of a geodesic ball is convex as a function of the ball volume, so Jensen's inequality preserves the mass-transplantation estimate after mollification, including a possible contribution at $t=0$.  Finally, we insert the radial comparison estimates into the Steklov Rayleigh quotient.

\subsection{Notation from Gu--Li--Wan \cite{GuLiWan2025}}

We view the hyperbolic space as a warped product
\begin{equation*}
    (\HH^n,g_{\HH})=([0,\infty)\times\Sn^{n-1},dr^2+\lambda(r)^2g_{\Sn^{n-1}}),
    \qquad
    \lambda(r)=\sinh r.
\end{equation*}
After choosing an origin $O$, let $B_r$ and $S_r$ denote the geodesic ball and sphere of radius $r$ centered at the origin. We consider the volume-area ratio function 
\begin{align*}
    g(r)&:=\frac{|B_r|}{|S_r|}=\frac{1}{\lambda(r)^{n-1}}\int_0^r\lambda(s)^{n-1}\,ds,
\end{align*}
Then the first positive Steklov eigenvalue of the geodesic ball $B_r\subset \mathbb{H}^n$ satisfies 
\begin{equation*}
    \sigma_1(B_r)=\frac{g'(r)}{g(r)}. 
\end{equation*}
Denote 
\begin{equation}\label{s4.q}
    q(r):=(g'(r))^2+(n-1)\frac{g(r)^2}{\lambda(r)^2}.
\end{equation}
The identities used below are
\begin{equation}\label{s4.q.div}
    g'=1-(n-1)\frac{\lambda'g}{\lambda},
    \qquad
    \diver(g\partial_r)=1,
    \qquad
    \diver(gg'\partial_r)=q.
\end{equation}
The analytic properties needed below are collected in
\cite[Propositions~3.4 and~3.5]{GuLiWan2025}: $g$ is nonnegative,
strictly increasing and concave, with
\begin{equation*}
    g(0)=0,
    \qquad g'(0)=\frac1n,
    \qquad \lim_{r\to\infty}g(r)=\frac1{n-1},
    \qquad \lim_{r\to\infty}g'(r)=0.
\end{equation*}
In particular, concavity makes $g'$ non-increasing and
$0<g'(r)\leq1/n$ for finite $r$.  The first identity in
\eqref{s4.q.div} also gives
\begin{equation*}
    \frac{\lambda'g}{\lambda}=\frac{1-g'}{n-1},
\end{equation*}
so this ratio is non-decreasing and
\begin{equation}\label{eq:hyp-a-range}
    \frac1n\leq \frac{\lambda'(r)g(r)}{\lambda(r)} \leq\frac1{n-1}.
\end{equation}
The radial function $x\mapsto g(r(x))$ is locally Lipschitz near the origin and $ \nabla g=g'(r)\partial_r$ 
almost everywhere.

For a finite-perimeter set $E\subset\HH^n$ define
\begin{align*}
    J(E):=\int_E\frac{\lambda'g}{\lambda}\,dv.
\end{align*}
 Set
\begin{equation*}
    J_B(r):=J(B_r),
    \qquad
    V_B(r):=|B_r|.
\end{equation*}
Following Gu--Li--Wan \cite[\S 4.1]{GuLiWan2025}, define $h:(0,\infty)\to(0,\infty)$ by
\begin{equation*}
    h(J_B(r))=\frac1{|S_r|}.
\end{equation*}
Indeed,
\begin{equation*}
    J_B'(r)
    =\frac{\lambda'(r)g(r)}{\lambda(r)}|S_r|>0,
\end{equation*}
and \eqref{eq:hyp-a-range} gives
\begin{equation*}
    J_B(r)\geq\frac1nV_B(r)\longrightarrow\infty
    \qquad\text{as }r\to\infty.
\end{equation*}
Hence $J_B$ is a smooth increasing bijection from $(0,\infty)$ onto
$(0,\infty)$. Consequently,
\begin{equation*}
    h(s)=\frac{1}{\left|S_{J_B^{-1}(s)}\right|}
\end{equation*}
is smooth on $(0,\infty)$. Lemma~4.1 of \cite{GuLiWan2025} states that $h$ is positive, decreasing and log-convex, and gives
\begin{equation}\label{eq:hyp-h-log-derivative}
    \frac{h'(J_B(r))}{h(J_B(r))}
    =-\frac{n-1}{V_B(r)}.
\end{equation}
For later use, log-convexity implies
\begin{equation}\label{eq:hyp-hhprime-monotone}
    \bigl(hh'\bigr)'=(h')^2+hh''\geq0,
\end{equation}
so $s\mapsto h(s)h'(s)$ is non-decreasing on $(0,\infty)$. Since $\lambda'g/\lambda$ is increasing, mass transplantation (see \cite[\S 4]{FL21} and \cite{We56}) gives
\begin{equation}\label{eq:hyp-J-mass-transplant}
    J(E)\geq J(B_r),
    \qquad \text{when}~~ |B_r|=|E|.
\end{equation}
The function $J_B$, viewed as a function of ball volume, is convex because
\begin{equation*}
    \frac{dJ_B}{dV_B}=\frac{\lambda'(r)g(r)}{\lambda(r)}
\end{equation*}
and $\lambda'g/\lambda$ is increasing.

\subsection{A weak transport inequality for \texorpdfstring{$\int_{\Sigma_t}g\dd\mu_t$}{int Sigma g}}

Let $\Omega\subset\HH^n$ be a bounded outward-minimizing domain with smooth boundary $\Sigma$.  Let $\{\Omega_t\}_{t\geq0}$ be the proper weak inverse mean curvature flow starting from $\Omega$.  We set $\Omega_0=\Omega$ and $\Sigma_0=\Sigma$.  The area growth law is
\begin{equation*}
    |\Sigma_t|=e^t|\Sigma|,
    \qquad t\geq0.
\end{equation*}

With respect to the chosen origin $O$, define, for $t>0$,
\begin{equation*}
    V(t):=|\Omega_t|,\qquad
    J(t):=\int_{\Omega_t}\frac{\lambda'g}{\lambda}\,dv,\qquad
    C(t):=\int_{\Omega_t}g'\,dv,
\end{equation*}
and, for almost every $t>0$, define
\begin{equation*}
    G(t):=\int_{\Sigma_t}g\,d\mu_t.
\end{equation*}
At $t=0$, prescribe
\begin{equation*}
    G(0):=\int_\Sigma g\,d\mu,\qquad
    V(0):=|\Omega|,\qquad
    J(0):=\int_\Omega\frac{\lambda'g}{\lambda}\,dv,\qquad
    C(0):=\int_\Omega g'\,dv.
\end{equation*}
The identity \eqref{s4.q.div} gives
\begin{equation*}
    C(t)=V(t)-(n-1)J(t),
    \qquad t\geq0.
\end{equation*}
Since the weak-flow sets $\Omega_t$ are nested and both
$\lambda'g/\lambda$ and $g'$ are nonnegative, the functions
$V$, $J$ and $C$ are non-decreasing on $[0,\infty)$. Whenever
distributional derivatives or time convolutions are used below, all four functions are extended to $t<0$ by their prescribed values at $t=0$.

\begin{lemma}\label{lem:hyp-weak-transport-G}
The function $G$ belongs to $L^1_{\mathrm{loc}}([0,\infty))$.  After the constant extensions described above, the following inequality holds in distributions on $\mathbb R$:
\begin{equation}\label{eq:hyp-measure-transport}
    D_tG\leq G\,dt+D_tC.
\end{equation}
Equivalently, for every nonnegative $\eta\in C_c^1([0,\infty))$,
\begin{equation}\label{eq:hyp-transport-test}
    -\int_0^\infty\eta'(t)G(t)\,dt-\eta(0)G(0)
    \leq
    \int_0^\infty\eta(t)G(t)\,dt
    -\int_0^\infty\eta'(t)C(t)\,dt-\eta(0)C(0).
\end{equation}
\end{lemma}

\begin{proof}
Apply Lemma \ref{lem:endpoint-transport} with $f=g$.  The function $G$ is locally integrable, and for every nonnegative $\eta\in C_c^1([0,\infty))$,
\begin{align}\label{s4-1.pf1}
    -\int_0^\infty\eta'G\,dt-\eta(0)G(0)
    &=\int_0^\infty\eta G\,dt
      +\int_{\HH^n\setminus\overline\Omega}
       \eta(u)g'(r)\langle\partial_r,\nu\rangle\,dv\nonumber\\
    &\quad+\eta(0)\int_\Sigma g(\tau-1)\,d\mu\nonumber\\
    &\leq\int_0^\infty\eta G\,dt
      +\int_{\HH^n\setminus\overline\Omega}\eta(u)g'(r)\,dv.
\end{align}
Here we used $\langle\partial_r,\nu\rangle\leq1$ and $\tau\leq1$.  On the other hand,
\begin{equation*}
    C(t)=C(0)+\int_{\HH^n\setminus\overline\Omega}
      \mathbf 1_{\{u<t\}}g'(r)\,dv.
\end{equation*}
Here $u\geq0$ on $\HH^n\setminus\Omega$, and \eqref{s4.q.div} and \eqref{eq:hyp-a-range} give $0\leq g'(r)\leq1/n$.  Therefore Fubini's theorem, applied as in \eqref{eq:Bk-distribution}, gives
\begin{align}\label{s4-1.pf2}
    -\int_0^\infty\eta'C\,dt-\eta(0)C(0)
    =&\int_{\HH^n\setminus\overline\Omega}g'(r)\left(-\int_0^\infty \eta'(t)\mathbf 1_{\{u<t\}}dt\right)dv\nonumber\\
   =& \int_{\HH^n\setminus\overline\Omega}\eta(u)g'(r)\,dv.
\end{align}
Equivalently,
\begin{equation*}
    D_tC
    =
    u_\#\left(
      g'(r)\,dv\mathbin{\llcorner}
      (\mathbb H^n\setminus\overline\Omega)
    \right).
\end{equation*}
Combining \eqref{s4-1.pf1} and \eqref{s4-1.pf2} proves
\eqref{eq:hyp-transport-test}, and hence
\eqref{eq:hyp-measure-transport}.
\end{proof}

\subsection{The weak Gu--Li--Wan monotonicity}
For almost every $t>0$, and at the separately prescribed value $t=0$, set
\begin{equation}\label{s4.M}
    \mathcal M(t):=\frac{1}{|\Sigma_t|}
    \frac{G(t)}{V(t)h(J(t))}.
\end{equation}
\begin{lemma}
\label{lem:hyp-D-estimate}
Extend $G$, $V$, $J$ and $C$ to $(-\infty,0)$ by their prescribed values at $t=0$. Then $G,V,J,C\in BV_{\mathrm{loc}}(\mathbb R)$.

Let $\rho_\varepsilon\in C_c^\infty(\mathbb R)$ be nonnegative and
satisfy
\begin{equation*}
    \int_{\mathbb R}\rho_\varepsilon(s)\,ds=1,
    \qquad
    \operatorname{supp}\rho_\varepsilon
    \subset(-\varepsilon,\varepsilon).
\end{equation*}
For $\Theta\in\{G,V,J,C\}$, define
\begin{equation*}
    \Theta_\varepsilon(t)
    :=(\rho_\varepsilon*\Theta)(t)
    =\int_{\mathbb R}\rho_\varepsilon(t-s)\Theta(s)\,ds.
\end{equation*}
Then $V_\varepsilon>0$, $J_\varepsilon>0$ and
$V_\varepsilon h(J_\varepsilon)>0$ on $\mathbb R$. Moreover, \begin{equation}\label{eq:hyp-mollified-measures}
\begin{aligned}
    0
    &\leq\frac1nV_\varepsilon'
    \leq J_\varepsilon'
    \leq\frac1{n-1}V_\varepsilon',\\
    C_\varepsilon'
    &=V_\varepsilon'-(n-1)J_\varepsilon'
    \geq0.
\end{aligned}
\end{equation}
Furthermore,
\begin{equation}\label{eq:hyp-mollified-D}
    \bigl(V_\varepsilon h(J_\varepsilon)\bigr)'
    \geq h(J_\varepsilon)C_\varepsilon',
\end{equation}
and
\begin{equation}\label{eq:hyp-mollified-G}
    G_\varepsilon'
    \leq G_\varepsilon+C_\varepsilon',
    \qquad
    G_\varepsilon\geq V_\varepsilon.
\end{equation}
\end{lemma}

\begin{proof}
We first compute the distributional derivatives of the extended
functions. Since $u\geq0$, the constant extensions of $V$ and $J$
satisfy, for every $t\in\mathbb R$,
\begin{align*}
    V(t)
    &=V(0)+\int_{\HH^n\setminus\overline\Omega}
      \mathbf 1_{\{u<t\}}\,dv,\\
    J(t)
    &=J(0)+\int_{\HH^n\setminus\overline\Omega}
      \mathbf 1_{\{u<t\}}
      \frac{\lambda'g}{\lambda}\,dv.
\end{align*}
For $t<0$, the two integrals vanish, so these formulas agree with the
constant extensions. On every bounded time interval, properness of $u$
restricts the spatial integrals to a compact set, while
\eqref{eq:hyp-a-range} bounds the weight $\lambda'g/\lambda$.
Therefore, Fubini's theorem gives
\begin{equation*}
    D_tV
    =u_\#\left(
      dv\mathbin{\llcorner}
      (\HH^n\setminus\overline\Omega)\right),
    \qquad
    D_tJ
    =u_\#\left(
      \frac{\lambda'g}{\lambda}\,dv
      \mathbin{\llcorner}
      (\HH^n\setminus\overline\Omega)\right).
\end{equation*}
These identities hold as locally finite Radon measures on $\mathbb R$.
In particular, the set
$\{u=0\}\setminus\overline\Omega$ is mapped by $u$ to $t=0$, so its
contribution is automatically included in the two measures.

The pointwise bounds \eqref{eq:hyp-a-range} now imply
\begin{equation}\label{equ-dtV}
    0\leq\frac1nD_tV
    \leq D_tJ
    \leq\frac1{n-1}D_tV,
    \qquad
    D_tC=D_tV-(n-1)D_tJ\geq0
\end{equation}
as inequalities of Radon measures on $\mathbb R$. Consequently, $V,J,C\in BV_{\mathrm{loc}}(\mathbb R)$. 

We next consider $G$. Let $\nu_{\Omega_t}$ denote the
measure-theoretic outward unit normal to $\Omega_t$, with
$\nu_{\Omega_0}:=\nu_\Omega$. Since $\diver(g\partial_r)=1$, the Gauss--Green formula gives
\begin{equation*}
    V(t)
    =\int_{\Sigma_t}
      g\langle\partial_r,\nu_{\Omega_t}\rangle\,d\mu_t
    \leq G(t)
\end{equation*}
for almost every $t>0$ and also at $t=0$. After the constant extension, $G\geq V$ almost everywhere on $\mathbb R$. Lemma \ref{lem:hyp-weak-transport-G} also gives \eqref{eq:hyp-measure-transport} in distributions on $\mathbb R$. Hence 
$$\zeta:=G\,dt+D_tC-D_tG$$
is a positive distribution and therefore a Radon measure. Since
\begin{equation*}
    D_tG=G\,dt+D_tC-\zeta,
\end{equation*}
and all three measures on the right are locally finite, $D_tG$ is
a locally finite signed Radon measure. Thus $G\in BV_{\mathrm{loc}}(\mathbb R)$.

Distributional differentiation commutes with convolution. Consequently,
\begin{equation*}
    \Theta_\varepsilon'=\rho_\varepsilon*D_t\Theta,
    \qquad \Theta\in\{G,V,J,C\}.
\end{equation*}
Since $\rho_\varepsilon$ is nonnegative, convolving
\eqref{equ-dtV} gives \eqref{eq:hyp-mollified-measures}. Convolving the transport inequality \eqref{eq:hyp-measure-transport} for $G$ gives
\begin{equation*}
    G_\varepsilon'\leq G_\varepsilon+C_\varepsilon',
\end{equation*}
while $G\geq V$ gives $G_\varepsilon\geq V_\varepsilon$. This proves
\eqref{eq:hyp-mollified-G}.

The formulas \eqref{equ-dtV} for $V$ and $J$ also show that
\begin{equation*}
    V(t)\geq V(0)>0,
    \qquad
    J(t)\geq J(0)\geq\frac1nV(0)>0
\end{equation*}
for every $t\in\mathbb R$. Therefore
$V_\varepsilon>0$, $J_\varepsilon>0$, and
$V_\varepsilon h(J_\varepsilon)>0$.

It remains to prove \eqref{eq:hyp-mollified-D}. View the ball quantity $J_B$ as a function of ball volume and set
\begin{equation*}
    F(\xi)=J_B\bigl(V_B^{-1}(\xi)\bigr).
\end{equation*}
The derivative of $F$ is $\lambda'g/\lambda$ at the corresponding radius, so $F$ is convex. Mass transplantation gives
\begin{equation*}
    J(t)\geq F(V(t)),\qquad t\in\mathbb R.
\end{equation*}
For $t<0$, this follows from the same inequality for the prescribed
initial domain. Jensen's inequality and the nonnegativity of $\rho_\varepsilon$ now yield
\begin{align*}
    J_\varepsilon(t)=&\int_{\mathbb{R}}\rho_\varepsilon(t-s)J(s)ds\\
    \geq &\int_{\mathbb{R}}\rho_\varepsilon(t-s)F(V(s))ds\\
    \geq &F\left(\int_{\mathbb{R}}\rho_\varepsilon(t-s)V(s)ds\right)\\
    =&F(V_\varepsilon(t))\\
    =&J_B\left(V_B^{-1}(V_\varepsilon(t))\right). 
\end{align*}
Since $V_B$ and $J_B$ are strictly increasing, this is equivalent to
\begin{equation*}
    V_B\left(J_B^{-1}(J_\varepsilon)\right)\geq V_\varepsilon.
\end{equation*}
Differentiating $V_\varepsilon h(J_\varepsilon)$ and using \eqref{eq:hyp-h-log-derivative} and \eqref{eq:hyp-mollified-measures}, we obtain
\begin{align*}
    \left( V_\varepsilon h(J_\varepsilon) \right)'=&V_\varepsilon'h(J_\varepsilon)+V_\varepsilon h'(J_\varepsilon)J_\varepsilon'\\
    =&h(J_\varepsilon)
      \left(V_\varepsilon'-(n-1)
      \frac{V_\varepsilon}{V_B(J_B^{-1}(J_\varepsilon))}J_\varepsilon'\right)\\
    \geq& h(J_\varepsilon)
      \left(V_\varepsilon'-(n-1)J_\varepsilon'\right)\\
     =&h(J_\varepsilon)C_\varepsilon'.
\end{align*}
This proves \eqref{eq:hyp-mollified-D}.
\end{proof}

\begin{proposition}[Monotonicity]\label{prop:hyp-GLW-monotonicity}
The quantity $\mathcal M(t)$ defined in \eqref{s4.M} agrees almost everywhere on $(0,\infty)$ with a non-increasing function.  At positive-time jump points we choose its right-continuous version, while at $t=0$ we retain the prescribed value in \eqref{s4.M}.
\end{proposition}

\begin{proof}
Define
\begin{equation*}
    R_\varepsilon
    :=\frac{G_\varepsilon}{V_\varepsilon h(J_\varepsilon)}.
\end{equation*}
The denominator is strictly positive by Lemma
\ref{lem:hyp-D-estimate}. Equations \eqref{eq:hyp-mollified-measures}-\eqref{eq:hyp-mollified-G} give
\begin{align*}
    R_\varepsilon'
    &=\frac{G_\varepsilon'}{V_\varepsilon h(J_\varepsilon)}
      -\frac{G_\varepsilon \left( V_\varepsilon h(J_\varepsilon) \right)'}{\left( V_\varepsilon h(J_\varepsilon) \right)^2}\\
    &\leq R_\varepsilon
      +\frac{C_\varepsilon'}{V_\varepsilon h(J_\varepsilon)}
      -\frac{G_\varepsilon h(J_\varepsilon)C_\varepsilon'}
       {V_\varepsilon^2h(J_\varepsilon)^2}\\
    &=R_\varepsilon
      +\frac{V_\varepsilon-G_\varepsilon}
       {V_\varepsilon^2h(J_\varepsilon)}C_\varepsilon'
    \leq R_\varepsilon.
\end{align*}
Hence $e^{-t}R_\varepsilon(t)$ is non-increasing on $\mathbb R$.

Let $\mathcal T\subset(0,\infty)$ be the common set of Lebesgue points of $G$, $V$, and $J$.  At every point of $\mathcal T$, the approximate-identity theorem gives convergence of $G_\varepsilon$, $V_\varepsilon$, and $J_\varepsilon$ to the corresponding original functions.  Since $V>0$, $J>0$ and $h$ is continuous on $(0,\infty)$, the quotients converge as well. If $s,t\in\mathcal T$ and $s<t$, first apply the monotonicity of $e^{-t}R_\varepsilon$ and then let $\varepsilon\downarrow0$ to obtain
\begin{equation*}
    e^{-s}\frac{G(s)}{V(s)h(J(s))}
    \geq
    e^{-t}\frac{G(t)}{V(t)h(J(t))}.
\end{equation*}
Thus $t\mapsto e^{-t}G(t)/(V(t)h(J(t)))$ agrees almost everywhere with a non-increasing representative on $(0,\infty)$. Here the choice of a representative only fixes the values at jump times and does not change the quantity almost everywhere. We choose the right-continuous version of the monotone function, while keeping the prescribed value at $t=0$. 

It remains to compare this prescribed endpoint value with the
positive-time values. Fix $a<0$ and $t\in\mathcal T$. For $0<\varepsilon<-a$, the convolutions at $a$ involve only the constant extensions, so
\begin{equation*}
G_\varepsilon(a)=G(0),\qquad
V_\varepsilon(a)=V(0),\qquad
J_\varepsilon(a)=J(0).
\end{equation*}
Monotonicity therefore implies, for $t\in\mathcal T$,
\begin{equation*}
e^{-a}\frac{G(0)}{V(0)h(J(0))}
\geq
e^{-t}
\frac{G_\varepsilon(t)}
{V_\varepsilon(t)h(J_\varepsilon(t))}.
\end{equation*}
With $a<0$ fixed, let $\varepsilon\downarrow0$ through values smaller
than $-a$. Then let $a\uparrow0$. In this order we obtain
\begin{equation*}
    \frac{G(0)}{V(0)h(J(0))}
    \geq e^{-t}\frac{G(t)}{V(t)h(J(t))},
    \qquad t\in\mathcal T.
\end{equation*}
Thus the non-increasing representative has the prescribed value at
$t=0$ and no separate right-limit value is inserted.

Finally, since $|\Sigma_t|=e^t|\Sigma|$, this proves that $\mathcal M(t)$ has a non-increasing representative on $[0,\infty)$.
\end{proof}

\begin{corollary}\label{cor:hyp-g2-lower}
For the initial domain $\Omega$,
\begin{equation}\label{eq:hyp-g2-lower}
    \int_\Sigma g^2\,d\mu
    \geq
    |\Sigma|\,|\Omega|^2h\left(\int_\Omega \frac{\lambda'g}{\lambda}\,dv\right)^2.
\end{equation}
\end{corollary}

\begin{proof}
We shall use the comparison principle \cite[Theorem~2.2]{HI2001} in the following form: if two weak flows start from bounded initial domains $E$ and $\Omega$ with $E\subset \Omega$, then the corresponding sublevel sets satisfy $E_t\subset \Omega_t$ for all $t\geq0$. 

Choose $p\in\Omega$ and a geodesic ball $B_{\rho_0}(p)$ compactly
contained in $\Omega$.  Its explicit spherical weak flow is
$B_{\rho(t)}(p)$, where $|S_{\rho(t)}|=e^t|S_{\rho_0}|$ and hence
$\rho(t)\to\infty$. By the comparison principle, we have $B_{\rho(t)}(p)\subset\Omega_t$ for all $t\geq 0$. Since $\Omega_t$ is open, every $x\in\partial^*\Omega_t$ satisfies
$d(p,x)\geq\rho(t)$. The triangle inequality then gives
\begin{equation*}
r(x)=d(O,x)\geq d(p,x)-d(O,p)\geq\rho(t)-d(O,p).
\end{equation*}
Therefore, taking the essential infimum over $\partial^*\Omega_t$ yields
\begin{equation}
    \operatorname*{ess\,inf}_{\partial^*\Omega_t}r
    \geq\rho(t)-d(O,p)
    \longrightarrow\infty.
    \label{equ-rinfty}
\end{equation}

Let $\mathcal T\subset(0,\infty)$ be a full-measure set consisting
of common Lebesgue points of $G$, $V$ and $J$, on which $\mathcal M(t)$ agrees with the non-increasing representative
provided by Proposition \ref{prop:hyp-GLW-monotonicity}. For such $t$, the preceding lower bound \eqref{equ-rinfty} for $r$ on $\partial^*\Omega_t$ yields
\begin{align}
\left|\frac{G(t)}{|\Sigma_t|}-\frac1{n-1}
\right|&\leq\operatorname*{ess\,sup}_{x\in\partial^*\Omega_t}
\left|g(r(x))-\frac{1}{n-1}\right|\nonumber\\
&\leq\sup_{s\geq\rho(t)-d(O,p)}
\left|g(s)-\frac{1}{n-1}\right|.  \label{equ-GoverSigma}
\end{align}
Since $\rho(t)-d(O,p)\to\infty$ and $g(s)\to1/(n-1)$ as $s\to\infty$, the right-hand side of \eqref{equ-GoverSigma} tends to zero. Therefore,
\begin{equation}
    \lim_{\substack{t\to\infty\\ t\in\mathcal T}}
    \frac{G(t)}{|\Sigma_t|}=\frac{1}{n-1}.
    \label{s4.lim}
\end{equation}

For $t\in\mathcal T$, let $R_t$ be the volume radius defined by $|B_{R_t}|=V(t)$. Since $B_{\rho(t)}(p)\subset\Omega_t$, we have
$V(t)\to\infty$, and therefore $R_t\to\infty$. From \eqref{eq:hyp-J-mass-transplant} and the monotonicity of $h$,
\begin{equation*}
    h(J(t))\leq h(J_B(R_t))=\frac1{|S_{R_t}|}.
\end{equation*}
Consequently,
\begin{equation*}
    \mathcal M(t)
    \geq
    \frac{|S_{R_t}|}{|B_{R_t}|}\frac{G(t)}{|\Sigma_t|}
    =\frac1{g(R_t)}\frac{G(t)}{|\Sigma_t|}.
\end{equation*}
Together with $g(R_t)\to1/(n-1)$ and \eqref{s4.lim}, this gives
\begin{equation*}
    \liminf_{\substack{t\to\infty\\t\in\mathcal T}}\mathcal M(t)\geq1.
\end{equation*}
Since $\mathcal M$ agrees on $\mathcal T$ with its non-increasing representative,
\begin{equation*}
\mathcal M(0)\geq\liminf_{\substack{t\to\infty\\ t\in\mathcal T}}\mathcal M(t)\geq 1.
\end{equation*}
Consequently,
\begin{equation*}
    \int_\Sigma g\,d\mu
    \geq
    |\Sigma|\,|\Omega|h\left(\int_\Omega \frac{\lambda'g}{\lambda}\,dv\right).
\end{equation*}
H\"older's inequality now gives
\begin{equation*}
    \int_\Sigma g^2\,d\mu
    \geq\frac1{|\Sigma|}\left(\int_\Sigma g\,d\mu\right)^2,
\end{equation*}
which proves \eqref{eq:hyp-g2-lower}.
\end{proof}

\subsection{Radial comparison estimate for $h(J)^2$}
\label{subsec:radial}

The weak-flow part above gives the lower bound \eqref{eq:hyp-g2-lower} for every $n\geq3$.  What remains is a purely radial comparison estimate for $\int_\Omega q\,dv$, where $q$ is defined in \eqref{s4.q}.  

For $n\geq5$, the following is the radial estimate underlying \cite[Lemmas~4.4 and~4.5]{GuLiWan2025}.

\begin{lemma}[The high-dimensional radial estimate]\label{lem:hyp-h2-high}
Assume $n\geq5$, and let $\Omega\subset\HH^n$ be bounded and measurable with $|\Omega|>0$.  Let $R$ be defined by $|B_R|=|\Omega|$, and set
\begin{equation*}
    J(\Omega):=\int_\Omega\frac{\lambda'g}{\lambda}\,dv.
\end{equation*}
Then
\begin{equation}\label{eq:hyp-h2-bound-nge5}
    h(J(\Omega))^2
    \geq
    \frac{1}{g'(R)|B_R|\,|S_R|^2}
    \int_\Omega q\,dv.
\end{equation}
\end{lemma}

\begin{proof}
Set $J_R=J(B_R)$. Since $\Omega$ and $B_R$ have the same volume and $\lambda'g/\lambda$ is increasing, mass transplantation gives $J(\Omega)\geq J_R$.  If $J(\Omega)=J_R$, set $\theta=J_R$. Otherwise, the mean value theorem gives $\theta\in(J_R,J(\Omega))$. In either case, there exists $\theta\in[J_R,J(\Omega)]$ such that
\begin{equation*}
    h(J(\Omega))^2 = h(J_R)^2 + 2h(\theta)h'(\theta)(J(\Omega)-J_R).
\end{equation*}
By \eqref{eq:hyp-hhprime-monotone}, the function $s\mapsto h(s)h'(s)$ is non-decreasing. Hence,
\begin{equation*}
    2h(\theta)h'(\theta) \geq 2h(J_R)h'(J_R)
    =-\frac{2(n-1)}{|B_R|\,|S_R|^2}.
\end{equation*}
It follows that
\begin{equation}\label{s4-4.hJ}
    h(J(\Omega))^2
    \geq
    \frac1{|S_R|^2}
    -\frac{2(n-1)}{|B_R|\,|S_R|^2}(J(\Omega)-J_R).
\end{equation}
On the ball, the divergence identity \eqref{s4.q.div} for $gg'\partial_r$ and the relation $g(R)|S_R|=|B_R|$ give
\begin{align}\label{s4-4.q}
    \int_{B_R}q\,dv
    =&\int_{B_R} \diver(gg'\partial_r)\,dv \nonumber\\
    =&\int_{S_R}gg'\,d\mu
    =g'(R)|B_R|.
\end{align}
For dimensions $n\geq 5$, the function
\begin{equation*}
    \Phi(r):=q(r)+\frac{2(n-1)}{n}\frac{\lambda'(r)g(r)}{\lambda(r)}
\end{equation*}
is decreasing by \cite[Lemma~4.4]{GuLiWan2025}.  Mass transplantation then gives $\int_{\Omega}\Phi(r)\mathrm{d}v\leq \int_{B_R}\Phi(r)\mathrm{d}v$. Therefore, 
\begin{align}\label{equ-JJR}
    \frac{2(n-1)}{n}\bigl(J(\Omega)-J_R\bigr)
   =&\int_\Omega \left(\Phi(r)-q(r)\right)\,dv-\int_{B_R} \left(\Phi(r)-q(r)\right)\,dv\nonumber\\
   \leq &  \int_{B_R}q\,dv-\int_\Omega q\,dv\nonumber\\
  = &
    g'(R)|B_R|-\int_\Omega q\,dv.
\end{align}
Finally, $g'(R)\leq1/n$, $J(\Omega)-J_R\geq0$ and \eqref{equ-JJR} imply
\begin{equation*}
    2(n-1)g'(R)(J(\Omega)-J_R)
    \leq g'(R)|B_R|-\int_\Omega q\,dv.
\end{equation*}
Substitution into the estimate \eqref{s4-4.hJ} proves \eqref{eq:hyp-h2-bound-nge5}.
\end{proof}

Gu's thesis \cite[Section~4.3.2]{GuThesis2026} gives a different argument for the low-dimensional cases $n=3,4$.  The idea is to apply a weighted mass-transplantation to a new decreasing function $\Psi$.  We record the details for convenience of readers.

\begin{lemma}[Weighted mass transplantation \cite{GuThesis2026,GLW26}]\label{lem:hyp-weighted-mass-transplant}
Let $B_R\subset \HH^n$ be the geodesic ball of radius $R$ centered at the chosen origin.  Let $E\subset\HH^n$ be bounded and measurable.  Let $f:[0,\infty)\to\mathbb R$ be non-increasing, and let $F\geq0$ be measurable.  Assume that $F$ and $|f|F$ belong to $L^1(E\cup B_R)$.  Then
\begin{equation}\label{eq:hyp-weighted-mt}
    \int_E f(r)F\,dv-\int_{B_R}f(r)F\,dv
    \leq
    f(R)\left(\int_EF\,dv-\int_{B_R}F\,dv\right).
\end{equation}
If $f$ is strictly decreasing and $F>0$ almost everywhere on $E\mathbin{\triangle}B_R$, equality forces $E=B_R$ up to a null set.
\end{lemma}

\begin{proof}
Equation \eqref{eq:hyp-weighted-mt} is equivalent to
\begin{equation*}
\int_E \bigl(f(r)-f(R)\bigr)F\,dv
\le
\int_{B_R}\bigl(f(r)-f(R)\bigr)F\,dv.
\end{equation*}
The difference between the left-hand side and the right-hand side is
\begin{equation*}
\int_{E\setminus B_R}\bigl(f(r)-f(R)\bigr)F\,dv
-
\int_{B_R\setminus E}\bigl(f(r)-f(R)\bigr)F\,dv.
\end{equation*}
On $E\setminus B_R$ one has $r\geq R$, hence $f(r)-f(R)\leq0$.  Since $F\geq0$, the first integral is non-positive.  On $B_R\setminus E$ one has $r\leq R$, hence $f(r)-f(R)\geq0$.  Thus the second integral is non-negative.  Therefore the displayed difference is non-positive, and \eqref{eq:hyp-weighted-mt} follows.

If $f$ is strictly decreasing and $F>0$ almost everywhere on $E\mathbin{\triangle}B_R$, equality can hold only if $E\setminus B_R$ and $B_R\setminus E$ both have measure zero.  Hence $E=B_R$ up to a null set.
\end{proof}

In particular, when $F\equiv 1$, Lemma \ref{lem:hyp-weighted-mass-transplant} reduces to the classical mass transplantation \cite{We56,FL21}.  Now we state the radial comparison estimate in low dimensions. 

\begin{lemma}[The radial comparison estimate in dimensions $3$ and $4$]\label{prop:hyp-h2-low}
Assume $n=3$ or $n=4$.  Let $\Omega\subset\HH^n$ be bounded and measurable with $|\Omega|>0$.  Let $R$ be defined by $|B_R|=|\Omega|$, and set
\begin{equation*}
    J(\Omega):=\int_\Omega \frac{\lambda'g}{\lambda}\,dv.
\end{equation*}
Then
\begin{equation}\label{eq:hyp-h2-bound-low}
    h(J(\Omega))^2
    \geq
    \frac{1}{g'(R)|B_R|\,|S_R|^2}
    \int_\Omega q\,dv.
\end{equation}
\end{lemma}

\begin{proof}
The proof of \eqref{eq:hyp-h2-bound-low} is similar to that of Lemma \ref{lem:hyp-h2-high}. Set $J_R=J(B_R)$. In dimensions $3$ and $4$, we still have the estimates $J(\Omega)\geq J_R$, \eqref{s4-4.hJ} and \eqref{s4-4.q}. A new point is that when $n=3$ or $n=4$, the function 
\begin{equation}\label{eq:hyp-Psi-def}
    \Psi(r):=g'(r)+(n-1)\frac{g(r)^2}{\lambda(r)^2g'(r)}
       +2(n-1)\frac{\lambda'(r)g(r)}{\lambda(r)}
\end{equation}
has a finite continuous extension to $[0,\infty)$ and is decreasing there (\cite[Section~4.3.2]{GuThesis2026}). We collect the details of the monotonicity in Appendix \ref{app:low-dimensional-calculation}. 

Since $\Psi$ is decreasing and $|\Omega|=|B_R|$, the mass transplantation (i.e., Lemma \ref{lem:hyp-weighted-mass-transplant} with weight $F\equiv1$) implies 
\begin{equation*}
    \int_\Omega \Psi\,dv\leq \int_{B_R}\Psi\,dv.
\end{equation*}
Now apply the weighted mass transplantation in Lemma \ref{lem:hyp-weighted-mass-transplant} to $f(r)=g'(r)$ with weight
\begin{equation*}
    F(r)=g'(r)+(n-1)\frac{g(r)^2}{\lambda(r)^2g'(r)}.
\end{equation*}
Here $f$ is decreasing because $g$ is concave, and $F\geq0$.  Since $g'F=q$ and
\begin{equation*}
    F=\Psi-2(n-1)\frac{\lambda'g}{\lambda},
\end{equation*}
we get
\begin{align}
    \int_\Omega q\,dv-\int_{B_R}q\,dv
    &\leq g'(R)\left(\int_\Omega F\,dv-\int_{B_R}F\,dv\right) \notag\\
    &=g'(R)\left(\int_\Omega \Psi\,dv-\int_{B_R}\Psi\,dv\right)
       -2(n-1)g'(R)(J(\Omega)-J_R) \notag\\
    &\leq -2(n-1)g'(R)(J(\Omega)-J_R).\label{eq:hyp-low-q-J}
\end{align}
Substituting \eqref{eq:hyp-low-q-J} and \eqref{s4-4.q} into \eqref{s4-4.hJ} gives
\begin{align*}
    h(J(\Omega))^2
    \geq &\frac1{|S_R|^2}
       -\frac{1}{g'(R)|B_R|\,|S_R|^2}
        \left(\int_{B_R}q\,dv-\int_\Omega q\,dv\right) \\
    =&\frac1{|S_R|^2}
       -\frac{1}{g'(R)|B_R|\,|S_R|^2}
        \left(g'(R)|B_R|-\int_\Omega q\,dv\right) \\
    =&\frac{1}{g'(R)|B_R|\,|S_R|^2}\int_\Omega q\,dv.
\end{align*}
This is \eqref{eq:hyp-h2-bound-low}.
\end{proof}

\subsection{Proof of the hyperbolic Weinstock inequality}

\begin{lemma}[Hyperbolic center of mass]\label{lem:hyp-center-of-mass}
Let $\Sigma\subset\HH^n$ be a smooth compact hypersurface.  There is a point $O\in\HH^n$ such that, in normal coordinates centered at $O$,
\begin{equation}\label{eq:hyp-center-mass}
    \int_\Sigma g(r)\frac{x_i}{r}\,d\mu=0,
    \qquad i=1,\ldots,n.
\end{equation}
The quotient is understood by its continuous extension at $r=0$.
\end{lemma}

\begin{proof}
Set $\mathcal A(r):=\int_0^r g(s)\,ds$ and define
\begin{equation*}
    \mathcal F(p):=\int_\Sigma\mathcal A\bigl(d(p,x)\bigr)\,d\mu(x).
\end{equation*}
Since $g(r)\to1/(n-1)$ as $r\to\infty$, the function $\mathcal F$ is proper and attains a minimum at some point $O$.  The expansion $g(r)=r/n+O(r^3)$ shows that the integrand is differentiable in $p$ even when $p=x$.  For every $v\in T_O\HH^n$, the first variation of distance gives
\begin{equation*}
    0=d\mathcal F_O(v)
    =-\int_\Sigma g(r)
      \left\langle v,\frac{\exp_O^{-1}(x)}{r}\right\rangle\,d\mu(x).
\end{equation*}
Writing $\exp_O^{-1}(x)=\sum_{i=1}^n x_i e_i$ in an orthonormal basis and taking $v=e_i$ proves \eqref{eq:hyp-center-mass}.
\end{proof}

\begin{proof}[Proof of Theorem \ref{thm:hyp-outward-weinstock}]
Choose the origin $O$ from Lemma \ref{lem:hyp-center-of-mass} and set
\begin{equation*}
    \varphi_i:=g(r)\frac{x_i}{r},
    \qquad i=1,\ldots,n.
\end{equation*}
The expansion $g(r)/r=1/n+O(r^2)$ shows that these functions extend smoothly across $O$.  Equation \eqref{eq:hyp-center-mass} makes them admissible in the Steklov Rayleigh quotient.  The polar-coordinate identities are
\begin{equation*}
    \sum_{i=1}^n\varphi_i^2=g^2,
    \qquad
    \sum_{i=1}^n|\nabla\varphi_i|^2
    =(g')^2+(n-1)\frac{g^2}{\lambda^2}=q.
\end{equation*}
Applying \eqref{eq:steklov-variational} to each $\varphi_i$ and summing gives
\begin{equation}\label{eq:hyp-steklov-upper}
    \sigma_1(\Omega)\int_\Sigma g^2\,d\mu
    \leq\int_\Omega q\,dv.
\end{equation}
Combining \eqref{eq:hyp-steklov-upper} with Corollary \ref{cor:hyp-g2-lower} gives
\begin{equation}\label{eq:hyp-sigma-basic}
    \sigma_1(\Omega)
    \leq
    \frac{\int_\Omega q\,dv}{|\Sigma|\,|\Omega|^2h(J(\Omega))^2}.
\end{equation}
Let $R$ be the volume radius defined by $|B_R|=|\Omega|$. We now use the radial comparison estimates from Subsection \ref{subsec:radial}.  These are \eqref{eq:hyp-h2-bound-nge5} for $n\geq5$ and \eqref{eq:hyp-h2-bound-low} for $n=3,4$.  Substituting the appropriate bound into \eqref{eq:hyp-sigma-basic} and using $|\Omega|=|B_R|$ gives
\begin{equation}\label{eq:hyp-area-sigma-all}
    |\Sigma|\sigma_1(\Omega)
    \leq
    \frac{g'(R)|B_R|\,|S_R|^2}{|B_R|^2}
    =|S_R|\frac{g'(R)}{g(R)}
    =|S_R|\sigma_1(B_R).
\end{equation}
Here we used $g(R)=|B_R|/|S_R|$ and $\sigma_1(B_R)=g'(R)/g(R)$.

Finally, let $R_*$ be defined by $|S_{R_*}|=|\Sigma|$.  The hyperbolic isoperimetric inequality gives $|S_R|\leq|\Sigma|=|S_{R_*}|$, hence $R\leq R_*$.  Lemma 3.8 of \cite{GuLiWan2025} states that $|S_r|\sigma_1(B_r)$ is non-decreasing.  Its proof also gives strict increase for $r>0$.  Indeed,
\begin{equation*}
    |S_r|\sigma_1(B_r)
    =|\Sn^{n-1}|\lambda(r)^{n-3}
      \frac{\lambda(r)^2g'(r)}{g(r)},
\end{equation*}
and
\begin{equation*}
    \frac{\lambda^2g'}{g}
    =\left(\frac{\lambda\lambda'g'}{g}\right)
      \left(\frac{\lambda}{\lambda'}\right).
\end{equation*}
The factor $\lambda^{n-3}$ is non-decreasing for $n\geq3$.  The first factor in the last product is positive and non-decreasing by the proof of \cite[Lemma~3.6]{GuLiWan2025}. The second factor is the strictly increasing function $\tanh r$. Therefore
\begin{equation*}
    |S_R|\sigma_1(B_R)
    \leq
    |S_{R_*}|\sigma_1(B_{R_*})
    =|\Sigma|\sigma_1(\Omega^*),
\end{equation*}
with equality only when $R=R_*$.  Combining this with \eqref{eq:hyp-area-sigma-all} and dividing by $|\Sigma|$ proves \eqref{eq:hyp-weinstock-final}.

If equality holds, then $R=R_*$.  Equality follows in the hyperbolic isoperimetric inequality, so $\Omega$ is a geodesic ball.  Conversely, geodesic balls attain equality.
\end{proof}


\appendix
\section{The low-dimensional radial calculation}\label{app:low-dimensional-calculation}

We give the explicit verification of the monotonicity of the function $\Psi$ defined in \eqref{eq:hyp-Psi-def},  following the calculation in Gu's thesis \cite[Section~4.3.2]{GuThesis2026}.

First take $n=3$.  Since $\lambda=\sinh r$,
\begin{equation}\label{eq:hyp-n3-g-gp}
    g(r)=\frac{\lambda\lambda'-r}{2\lambda^2},
    \qquad
    g'(r)=\frac{r\lambda'-\lambda}{\lambda^3}.
\end{equation}
Substituting \eqref{eq:hyp-n3-g-gp} into \eqref{eq:hyp-Psi-def}, and using $(\lambda')^2=\lambda^2+1$, gives
\begin{equation}\label{eq:hyp-Psi-n3-rational}
    \Psi(r)=\frac{U(r)}{V(r)},
\end{equation}
where
\begin{align*}
    U(r)&=4r\lambda(\lambda')^3-3\lambda^2(\lambda')^2
          -2r^2(\lambda')^2-2r\lambda\lambda'+2\lambda^2+r^2,
          \\
    V(r)&=2\lambda^3(r\lambda'-\lambda).
\end{align*}
A direct differentiation of \eqref{eq:hyp-Psi-n3-rational}, followed by the substitution $(\lambda')^2=\lambda^2+1$, gives the following identity:
\begin{equation}\label{eq:hyp-Psi-prime-N}
    \Psi'(r)=\frac{\lambda(r)^2N(r)}{2(r\lambda'(r)-\lambda(r))^2},
\end{equation}
with
\begin{align*}
    N(r)&=-\frac{\lambda'}{\lambda}
       +\frac{11r+4r^3}{\lambda^2}
       -\frac{(3+10r^2)\lambda'}{\lambda^3}
       +\frac{9r+6r^3}{\lambda^4}
       -\frac{9r^2\lambda'}{\lambda^5}
       +\frac{3r^3}{\lambda^6}.
\end{align*}
The reason for introducing $N$ is that its derivative factors with a definite sign:
\begin{equation*}
    N'(r)=
    -\frac{2}{\lambda^7}(r\lambda'-\lambda)
      \left(r\bigl(2+\lambda^2+(\lambda')^2\bigr)-3\lambda\lambda'\right)^2
      \leq0.
\end{equation*}
Moreover $r\lambda'-\lambda>0$ for $r>0$, since its derivative is $r\lambda>0$ and it vanishes at $r=0$.  Expanding to the first nonzero order gives
\begin{equation*}
    N(r)=-\frac{32}{4725}r^7+O(r^9).
\end{equation*}
In particular, $N(r)\to0$ as $r\downarrow0$.  Since $N'\leq0$, it follows that $N(r)\leq0$ for every $r>0$.  Equation \eqref{eq:hyp-Psi-prime-N} now gives $\Psi'(r)\leq0$ for $n=3$.

For $n=4$, the radial function is explicit:
\begin{equation*}
    g(r)=\frac{\lambda(\lambda'+2)}{3(\lambda'+1)^2},
    \qquad
    g'(r)=\frac{1}{(\lambda'+1)^2}.
\end{equation*}
Substitution into \eqref{eq:hyp-Psi-def} gives
\begin{equation*}
    \Psi(r)=\frac13\left(7+\frac{2}{\lambda'+1}
          -\frac{2}{(\lambda'+1)^2}\right)
       =\frac52-\frac16\left(1-\frac{2}{\lambda'+1}\right)^2.
\end{equation*}
Since $\lambda'(r)=\cosh r$ is increasing, the last expression is decreasing.  The expansions $g(r)=r/n+O(r^3)$ and $g'(r)=1/n+O(r^2)$ show that $\Psi$ has a finite continuous extension at the origin.

This completes the verification of the monotonicity of $\Psi$ in
dimensions $3$ and $4$.

\section*{Acknowledgements}
This work was supported by the National Key Research and Development Program of China (2021YFA1001800), the National Natural Science Foundation of China (12531002), and the Fundamental Research Funds for the Central Universities.  The third author was also supported by the China Postdoctoral Science Foundation (2025M783146).

\end{document}